\documentclass[11pt,a4paper]{article}
\usepackage{amssymb}
\usepackage{amsmath, amsthm, amscd, amsfonts, amssymb, graphicx, color}
\usepackage[bookmarksnumbered, colorlinks, plainpages]{hyperref}
\newtheorem{thm}{Theorem}[section]

\newtheorem{lem}[thm]{Lemma}
\newtheorem{prob}[thm]{Problem}

\newtheorem{prop}[thm]{Proposition}
\newtheorem{defn}[thm]{Definition}

\numberwithin{equation}{section}
\numberwithin{equation}{subsection}

\begin{document}

\title{Total dominator total chromatic numbers of\\cycles and paths}

\author{$^{1}$Adel P. Kazemi and $^{2}$Farshad Kazemnejad
 \\[1em]
$^{1}$ Department of Mathematics, \\ University of Mohaghegh Ardabili, Ardabil, Iran\\
ORCID: 0000-0002-6957-733X\\
Email: adelpkazemi@yahoo.com\\
[1em]
$^{2}$ Department of Mathematics, School of Sciences\\ Ilam University, Ilam, Iran\\
Email:  kazemnejad.farshad@gmail.com\\
}

\maketitle

\begin{abstract}
The total dominator total coloring of a graph is a total coloring of the graph such that each object (vertex or edge) of the graph is adjacent or incident to every object of some color class. The minimum namber of the color classes of a total dominator total coloring of a graph is called the total dominator total chromatic number of the graph. In \cite{KKM2019}, the authors initiated to study the total dominator total coloring of a graph and found some useful results, and presented some problems. Finding the total dominator total chromatic numbers of cycles and paths were two of them which we consider them here.
\\[0.2em]

\\[0.2em]

\end{abstract}

\pagestyle{myheadings}
\markboth{\centerline {\scriptsize A. P. Kazemi and F. Kazemnejad}}     {\centerline {\scriptsize A. P. Kazemi and F. Kazemnejad,~~~~~~~~~~~~~~~~~~~~~~Total dominator total chromatic numbers of cycles and paths}}


\section{\bf Introduction}

All graphs considered here are non-empty, finite, undirected and simple. For
standard graph theory terminology not given here we refer to \cite{West}. Let $%
G=(V,E) $ be a graph with the \emph{vertex set} $V$ of \emph{order}
$n(G)$ and the \emph{edge set} $E$ of \emph{size} $m(G)$. The
\emph{open neighborhood} and the \emph{closed neighborhood} of a
vertex $v\in V$ are $N_{G}(v)=\{u\in V\ |\ uv\in E\}$ and
$N_{G}[v]=N_{G}(v)\cup \{v\}$, respectively. The \emph{degree} of a
vertex $v$ is also $deg_G(v)=\vert N_{G}(v) \vert $. The
\emph{minimum} and \emph{maximum degree} of $G$ are denoted by
$\delta =\delta (G)$ and $\Delta =\Delta (G)$, respectively. If
$\delta (G)=\Delta (G)=k$, then $G$ is called $k$-\emph{regular}. 
 An \emph{independent set} of $G$ is a subset of vertices of $G$, no two of which are adjacent. And a \emph{maximum independent set} is an independent set of the largest cardinality in $G$. This cardinality is called the \emph{independence number} of $G$, and is denoted by $\alpha(G)$. Also a \emph{mixed independent set} of $G$ is a subset of $V(G)\cup E(G)$, no two objects of which are adjacent or incident, and a \emph{maximum mixed  independent set} is a mixed independent set of the largest cardinality in $G$. This cardinality is called the \emph{mixed independence number} of $G$, and is denoted by $\alpha_{mix}(G)$. Two isomorphic graphs $G$ and $H$ are shown by $G\cong H$.

\vskip 0.2 true cm

We write $C_{n}$ and $P_{n}$ for a \emph{cycle} and a \emph{path} 
of order $n$, respectively, while 
   $G[S]$   is \emph{induced subgraph} of $G$ by a vertex set $S$. The \emph{line graph} $L(G)$ of $G$ is a graph with the vertex set $E(G)$ and two vertices of $L(G)$ are adjacent when they are incident in $G$. 
The \emph{total graph} $T(G)$ of a graph $G=(V,E)$ is the graph whose vertex set is $V\cup E$ and two vertices are adjacent whenever they are either adjacent or incident in $G$ \cite{Behzad}. It is obvious that if $G$ has order $n$ and size $m$, then $T(G)$ has order $n+m$ and size $3m+|E(L(G))|$, and also $T(G)$ contains both $G$ and $L(G)$ as two induced subgraphs and it is the largest graph formed by adjacent and incidence relation between graph elements. Since $deg_{T(G)}(v_i)=2deg_G(v_i)$ and $deg_{T(G)}(e_{ij})=deg_G(v_i)+deg_G(v_j)$, if $G$ is $k$-regular, then $T(G)$ is $2k$-regular. Also we have $\alpha_{mix}(G)=\alpha(T(G))$. 

\vskip 0.2 true cm

Here, we fix a notation for the vertex set and the edge set of line and  total of a graph which we use thorough this paper. For a graph $G=(V,E)$ with the vertex set $V=\{v_i|\ 1\leq i\leq n\}$, we have $V(L(G))=\mathcal{E}$ and $E(L(G))=\{e_{ij}e_{ik}~|~ e_{ij}, e_{ik}\in \mathcal{E} \mbox{ and } j\neq k\}$, $V(T(G))=V\cup \mathcal{E}$ and $E(T(G)) = E\cup E(L(G)) \cup \{e_{ij}v_i,e_{ij}v_j~|~e_{ij}\in \mathcal{E} \}$,  where $\mathcal{E}=\{e_{ij}~|~v_{i}v_{j}\in E\}$. In Figure \ref{TDTC1} a graph $G$ and its total graph are shown for an example.
\begin{figure}[ht]
\centerline{\includegraphics[width=8cm, height=3cm]{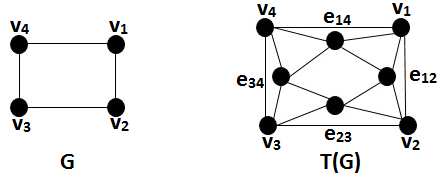}}
\vspace*{-0.2cm}
\caption{ The illustration of $G$ (left) and $T(G)$ (right).}\label{TDTC1}
\end{figure}

\vskip 0.2 true cm

\textbf{DOMINATION.} Domination in graphs is now well studied in graph theory and the literature
on this subject has been surveyed and detailed in the two books by Haynes,
Hedetniemi, and Slater~\cite{hhs1, hhs2}. A famous type of domination is total domination, and the literature on this subject has been surveyed and detailed in the recent
book~\cite{HeYe13}. A \emph{total dominating set}, briefly TDS, $S$ of a graph $G=(V,E)$ is a subset
of the vertex set of $G$ such that for each vertex $v$, $N_G(v)\cap S\neq \emptyset$. The \emph{total domination number $\gamma_{t}(G)$} of $G$ is the minimum cardinality of a TDS of $G$. Similarly, a subset $S\subseteq V\cup E$ of a graph $G$ is called a \emph{total mixed dominating set}, briefly TMDS, of $G$ if each object of $(V\cup E)$ is either adjacent or incident to an object of $S$, and the \emph{total mixed domination number} $\gamma_{tm}(G)$ of $G$ is the minimum cardinality of a TMDS \cite{KK2017}. A min-TDS/ min-TMDS of $G$ denotes a TDS/ TMDS of $G$ with minimum cardinality. Also we agree that \emph{a vertex $v$ dominates an edge} $e$ or \emph{an edge $e$ dominates a vertex} $v$ mean $v\in e$. Similarly, we agree that \emph{an edge dominates another edge} means they have a common vertex.  The next theorem can be easily obtained.

\begin{thm} \emph{\cite{KK2017}}
\label{gamma_{tm}(G)=gamma_{t}(T(G))}
For any graph $G$ without isolate vertex, $\gamma_{tm}(G)=\gamma_{t}(T(G))$.
\end{thm}
\vskip 0.2 true cm
\textbf{GRAPH COLORING.} Graph coloring is used as a model for a vast
number of practical problems involving allocation of scarce
resources (e.g., scheduling problems), and has played a key role in
the development of graph theory and, more generally, discrete
mathematics and combinatorial optimization. Graph colorability
is NP-complete in the general case, although the problem is solvable
in polynomial time for many classes \cite{GJ}. A \emph{proper coloring} of a graph $G$ is a function from
the vertices of the graph to a set of colors such that any two adjacent vertices have different colors, and the minimum number of colors needed in a proper
coloring of a graph is called the \emph{chromatic number} $\chi (G)$ of $G$. In a simlar way, a \emph{total coloring} of $G$ assigns a color to each vertex and to each edge so that colored objects have different colors when they are adjacent or incident, and the minimum number of colors needed in a total coloring of a graph is called the \emph{total chromatic number} $\chi_{T}(G)$ of $G$ \cite{West}. 
 A \emph{color class} in a coloring of a graph is a set consisting of all those objects assigned the same color. For simply, if $f$ is a coloring of $G$ with the coloring classes $V_1$, $V_2$, $\cdots$ , $V_{\ell}$, we write $f=(V_1,V_2,\cdots,V_{\ell})$. Motivated by the relation between coloring and total dominating, the concept of total dominator coloring in graphs introduced in \cite{Kaz2015} by Kazemi, and extended in \cite{Hen2015,Kaz-Par,Kaz2014,Kaz2016,KK2018}.

\begin{defn} 
\label{total dominator coloring} \emph{ \cite{Kaz2015} A} total dominator coloring,
\emph{briefly TDC, of a graph $G$ with a possitive minimum degree is a proper coloring of $G$ in
which each vertex of the graph is adjacent to every vertex of some
color class. The }total dominator chromatic number $\chi_{d}^t(G)$
\emph{of $G$ is the minimum number of color classes in a TDC of $G$.}
\end{defn}

In \cite{KKM2019}, the authors initiated studying of a new concept called total dominator total coloring in graphs which is obtained from the concept of total dominator coloring of a graph by replacing total coloring of a graph instead of (vertex) coloring of it.
\begin{defn} 
\label{total dominator total coloring} \emph{\cite{KKM2019} A} total dominator total coloring,
\emph{briefly TDTC, of a graph $G$ with a possitive minimum degree is a total coloring of $G$ in
which each object of the graph is adjacent or incident to every object of some
color class. The }total dominator total chromatic number $\chi_d^{tt}(G)$
\emph{of $G$ is the minimum number of color classes in a TDTC of $G$.}
\end{defn}

It can be easily obtained the next theorem.

\begin{thm}
\emph{\cite{KKM2019}}
\label{chi_d^{tt}(G)=chi_d^{t}(T(G))}
For any graph $G$ without isolate vertex, $\chi_d^{tt}(G)=\chi_d^{t}(T(G))$.
\end{thm}
For any TDC (TDTC) $f=(V_1,V_2,\cdots,V_{\ell})$ of a graph $G$, a vertex (an object) $v$ is called a \emph{common neighbor} of $V_i$ or we say $V_i$ \emph{totally dominates} $v$, and we write $v\succ_t V_i$, if vertex (object) $v$ is adjacent (adjacent or incident) to every vertex (object) in $V_i$. Otherwise we write $v \not\succ_t  V_i$. The set of all common neighbors of $V_i$ with respect to $f$ is called the \emph{common neighborhood} of $V_i$ in $G$ and denoted by $CN_{G,f}(V_i)$ or simply by $CN(V_i)$. Also every TDC or TDTC of $G$  with $\chi_d^t(G)$ or $\chi_d^{tt}(G)$ color classes is called  a \emph{min}-TDC or a \emph{min}-TDTC, respectively.
 For some examples see Figure \ref{TDTC11}.
\begin{figure}[ht]\label{TDTC11}
\centerline{\includegraphics[width=12.5cm, height=2cm]{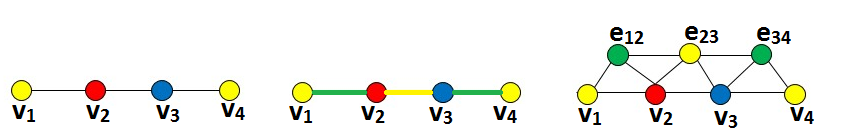}}
\vspace*{0cm}
\caption{A min-TDC of $P_4$ (left) and a min-TDTC of $P_4$ (Middle) with its corresponding min-TDC of $T(P_4)$ (right).} \label{TDTC11}
\end{figure}

Also for any TDC $(V_1,V_2,\cdots , V_{\ell})$ and any TDTC $(W_1,W_2,\cdots , W_{\ell})$  of a graph $G$, we have
\begin{equation}
\label{CN(V_i)=V}
\bigcup_{i=1}^{\ell} CN(V_i)=V(G) \mbox{ and } \bigcup_{i=1}^{\ell} CN(W_i)=V(G)\cup E(G).
\end{equation}

\vspace{0.2cm}

\textbf{GOAL.} In \cite{KKM2019}, the authors initiated to study the total dominator total coloring of a graph and found some useful results, and presented some problems. Finding the total dominator total chromatic numbers of cycles and paths were two of them which we consider them here.

\vspace{0.15cm}

We recall the following proposition from \cite{Kaz2015} which is useful for our investigation. Propositions \ref{chi^{t}_d T(C_n)} and \ref{chi^{t}_d T(P_n)} show that the upper bound given in Proposition \ref{chi_d^t =<g_t+min{chi(G-S)}} is tight.
\begin{prop}\emph{\cite{Kaz2015}}
\label{chi_d^t =<g_t+min{chi(G-S)}} For any connected graph $G$ with $\delta(G)\geq 1$, 
\begin{equation}\label{chi_d^t(G) =<gama_t(G)+ min chi(G[V(G)-S])}
\chi_d^t(G) \leq \gamma_{t}(G)+ \min_S \chi(G[V(G)-S]),
\end{equation}
where $S\subseteq V(G)$ is a min-TDS of $G$. And so $\chi_d^t(G) \leq \gamma_t(G)+ \chi(G)$.
\end{prop}
\section{Cycles}
Here, we calculate the total dominator total chromatic number of cycles. First we we recall a proposition from \cite{KK2017} and calculate the mixed indepence number of a cycle.
\begin{prop} \emph{\cite{KK2017}}
\label{gamma_t T(C_n)} For any cycle $C_n$ of order $n\geq 3$,
\begin{equation*}
\gamma_{tm}(C_n)=\left\{
\begin{array}{ll}
4\lceil n/7\rceil -3 & \mbox{if }n\equiv 1 \pmod{7},\\
4\lceil n/7\rceil -2 & \mbox{if }n\equiv 2,3 \pmod{7},\\
4\lceil n/7\rceil -1 & \mbox{if }n\equiv 4 \pmod{7},\\
4\lceil n/7\rceil  & \mbox{if }n\equiv 0,5,6 \pmod{7},
\end{array}\right.
\end{equation*}
or equivalently 
\begin{equation*}
\gamma_{tm}(C_n)=\left\{
\begin{array}{ll}
\lceil \frac{4n}{7}\rceil +1 & \mbox{if }n\equiv 5 \pmod{7},\\
\lceil \frac{4n}{7}\rceil     & \mbox{if }n\not\equiv 5 \pmod{7}.
\end{array}\right.
\end{equation*}
\end{prop}
\begin{lem}\label{alpha (T(C_n))=lfloor 2n/3 rfloor}
For any cycle $C_n$ of order $n\geq 3$, $\alpha_{mix}(C_n)=\lfloor \frac{2n}{3}\rfloor$.
\end{lem}
\begin{proof} 
Let  $C_{n}:v_{1}v_{2}\cdots v_{n}$ be a cycle of order $n\geq 3$. Then $V(T(C_{n}))=V\cup \mathcal{E}$ where  $\mathcal{E}=\{e_{i(i+1)}~|~1\leq i \leq n \}$. Since every vertex in an independent set belongs to exaxtly three tiangles and also $T(C_n)$ has $2n$ distinct triangles, we obtain $\alpha (T(C_n))\leq \lfloor 2n/3\rfloor$. On the other hand, since 
\begin{equation*}
\begin{array}{ll}
\{v_{3i-2},e_{(3i-1)(3i)}~|~ 1\leq i \leq \lfloor \frac{n}{3} \rfloor  \} & \mbox{ when }n\equiv 0,1 \pmod{3},\\ 
\{v_{3i-2}~|~ 1\leq i \leq \lceil \frac{n}{3} \rceil  \}\cup \{e_{(3i-1)(3i)}~|~ 1\leq i \leq \lfloor \frac{n}{3} \rfloor  \} & \mbox{ when }n\equiv 2 \pmod{3},
\end{array}
\end{equation*}
are independent sets of cardinality $\lfloor 2n/3\rfloor$, we have $\alpha_{mix}(C_n)=\alpha (T(C_n))=\lfloor \frac{2n}{3}\rfloor$. 
\end{proof}
\begin{prop}
\label{chi^{t}_d T(C_n)}
For any cycle $C_n$ of order $n\geq 3$,
\begin{equation*}
\chi^{tt}_d(C_n)=\left\{
\begin{array}{ll}
\gamma_{tm}(C_n)+1  & \mbox{if }n=3,4,5 \\
\gamma_{tm}(C_n)+2  & \mbox{if }n=6,9,12 \\
\gamma_{tm}(C_n)+3  & \mbox{if }n\geq 7 ~\mbox{and}~n \neq 9,12.
\end{array}
\right.
\end{equation*}
which by considering Proposition \ref{gamma_t T(C_n)} implies  
\begin{equation*}
\chi^{tt}_d(C_n)=\left\{
\begin{array}{ll}
n & \mbox{if }3\leq n \leq 8,\\
n-1 & \mbox{if }n=9,
\end{array}\right.
\end{equation*}
and for $n\geq 10$,
\begin{equation*}
\chi^{tt}_d(C_n)=\left\{
\begin{array}{ll}
\lceil \frac{4n}{7}\rceil +4 & \mbox{if }n\equiv 5 \pmod{7} \mbox{ and } n\neq 12,\\
\lceil \frac{4n}{7}\rceil +3 & \mbox{if }n\not\equiv 5 \pmod{7} \mbox{ or } n=12.
\end{array}\right.
\end{equation*}
\end{prop}

\begin{proof}
Let $C_{n}:v_{1}v_{2}\cdots v_{n}$ be a cycle of order $n\geq 3$. Then $V(T(C_{n}))=V\cup \mathcal{E}$ where $\mathcal{E}=\{ e_{i(i+1)}~|~1\leq i \leq n \}$. We know from \cite{KK2017} that for $n \geq 3$, the sets
\begin{equation*}
\begin{array}{ll}
S_0=\{v_{7i+2}, v_{7i+3}, e_{(7i+5)(7i+6)}, e_{(7i+6)(7i+7)}~|~ 0 \leq i \leq \lfloor n/7\rfloor-1\} & \mbox{if }n \equiv 0 \pmod{7}, \\
S=S_0 \cup\{e_{(n-1)n}\}& \mbox{if }n \equiv 1 \pmod{7}, \\
S=S_0 \cup\{v_{n-1},v_{n}\} & \mbox{if }n \equiv 2,3 \pmod{7}, \\
S=S_0 \cup\{v_{n-2},v_{n-1},v_n\} & \mbox{if }n \equiv 4 \pmod{7}\\
S=S_0 \cup\{v_{n-3},v_{n-2},v_{n-1},v_n\} & \mbox{if }n \equiv 5 \pmod{7},\\
S=S_0 \cup\{v_{n-4},v_{n-3},e_{(n-2)(n-1)},e_{(n-1)n}\} & \mbox{if }n \equiv 6 \pmod{7}.
\end{array}
\end{equation*}
are min-TDSs of $T(C_n)$, and since $\chi (T(C_n)-S_r)\leq 3$ for $0\leq r\leq 6$, we have 
\begin{equation}
\label{chi^{t}_d (T(C_n))  leq gamma_t(T(C_n))+3}
\chi^{tt}_d (C_n)\leq \gamma_{tm}(C_n)+3,
\end{equation} 
by  Proposition \ref{chi_d^t =<g_t+min{chi(G-S)}}. For $n\geq 20$ or $n=15,18$, since $T(C_n)$ has $2n$ distinct triangles with the vertex sets $\{v_i, v_{i+1}, e_{i(i+1)}\}$  or $\{e_{i(i+1)}, v_{i+1}, e_{(i+1)(i+2)}\}$ for $1\leq i \leq n$, and at least two color classes are needed for totally dominating the vertices of the consecutive tiangles with the vertex sets $\{ v_{i}, v_{i+1}, v_{i+2}, v_{i+3}, e_{i(i+1)}, e_{(i+1)(i+2)}, e_{(i+2)(i+3)} \}$  or $\{v_{i+1}, v_{i+2},v_{i+3}, e_{i(i+1)}, e_{(i+1)(i+2)}, e_{(i+2)(i+3)},e_{(i+3)(i+4)}\}$, we conclude that the number of used color classes in  $T(C_n)$ is at least $2 \lfloor \frac{2n}{5} \rfloor \geq \gamma_{t}(T(C_{n}))+3 $, and so $\chi^{tt}_d(C_n)=\chi_d^t(T(C_n))=\gamma_{tm}(C_{n})+3 $ by (\ref{chi^{t}_d (T(C_n))  leq gamma_t(T(C_n))+3}). Therefore, by considering the following facts in which $f=(V_{1},V_{2},\cdots,V_{\ell})$ is an arbitrary min-TDC of $T(C_{n})$, $|V_1|\geq |V_2|\geq \cdots \geq |V_{\ell}|$, $A_{i}=\{V_{k}~|~ v \succ_{t} V_{k}\mbox{ and }|V_{k}| = i \mbox{ for some }v \in V\cup \mathcal{E}\}$ and $|A_{i}|=a_i$ for $1 \leq i \leq \alpha$ where $\alpha:=\alpha_{mix}(C_n)=\lfloor \frac{2n}{3} \rfloor $, we continue our proof when $3 \leq n \leq 19$ except $n=15,18$.
\begin{itemize}
\item[$\star$] \texttt{Fact 1.} $|V_{k}|\leq \lfloor \frac{2n}{3}\rfloor $ for $1\leq k\leq \ell$, and $\sum_{i=1}^{\ell}|V_i|=2n$.
 
\item[$\star$] \texttt{Fact 2.} For any $v \in V\cup \mathcal{E}$, if $v \succ_{t} V_{k}$ for some $1\leq k\leq \ell$, then $|V_{k}|\leq 2$.

\item[$\star$] \texttt{Fact 3.} For any vertex  $v \in M$, if $v \succ_{t} V_{k}$ for some $1\leq k\leq \ell$ and $|V_k|=2$, then $CN(V_k)\cap M=\{v\}$, where $M\in \{\mathcal{E},V\}$, and since $CN(V_k)\cap V\neq \emptyset$ if and only if $CN(V_k)\cap \mathcal{E}\neq \emptyset$, we have $|CN(V_k)|=2$.

\item[$\star$] \texttt{Fact 4.} For any color class $V_k$ of cardinality one, $|CN(V_k)\cap V|=|CN(V_k)\cap \mathcal{E}|=2$.

\item[$\star$] \texttt{Fact 5.} $2a_1+a_2\geq n$ (by \texttt{Facts} 3, 4).

\item[$\star$] \texttt{Fact 6.} $\gamma_{tm}(C_n)\leq a_1+a_2\leq \ell $. Because the set $S$ is a TDS of $T(C_n)$ where $|S\cap V_i|=1$ for each $V_i\in A_1\cup A_2$ by \texttt{Fact} 2 (for left), and $a_1+\cdots +a_{\alpha}=\ell$ (for right).

\item[$\star$] \texttt{Fact 7.} $n-\ell \leq a_1 \leq \lfloor \frac{\alpha \ell -2n}{\alpha -1} \rfloor$. Because
\begin{equation*}
\begin{array}{llll}
2n-a_1             &   =   & |V(T(C_n))|-|\{V_{i} ~|~|V_{i}|=1 \mbox{ for } 1\leq i \leq \ell\}| \\
                          & =     & \Sigma_{|V_i|\geq 2} |V_i|\\
                          & \leq & (\ell-a_1)\alpha
     \end{array}
\end{equation*} 
implies the upper bound, and for the lower bound
\begin{equation}
\begin{array}{llll}
2\ell -n             &   \geq    & 2(a_1+a_2)-n & (\mbox{by } \texttt{Fact } 6) \\
                       &  \geq     & a_2                & (\mbox{by }  \texttt{Fact } 5),
     \end{array}
\end{equation} 
implies 
\begin{equation*}
\begin{array}{llll}
a_1             &  \geq   & \frac{n-a_2}{2} & (\mbox{by }  \texttt{Fact } 5) \\
                          & \geq  & n-\ell & \mbox{by } (3.0.2).
     \end{array}
\end{equation*} 

\item[$\star$] \texttt{Fact 8.} $\max\{0,n-2a_1,\gamma_{tm}(C_n)-a_1\}\leq a_2\leq \min\{ 2\ell-n, \ell -a_1\}$ (by \texttt{Facts} $5-7$).
\item[$\star$]  \texttt{Fact 9.} If $J=\{k~|~|V_k|=2 \mbox{ and } |CN(V_k)|\neq 0 \}$, then the number of isolate vertices of $T(C_n)[\bigcup_{|V_k|=1}V_k]$ is at most $|J|$ (because $(|V_k|,|CN(V_k)|)=(2,2)$ implies $T(C_n)[CN(V_k)]\cong K_2$, by \texttt{Fact 3}). 
\item[$\star$] \texttt{Fact 10.}  If $|\bigcup_{|V_i|=1} CN(V_i)|=4a_1$, then $a_1\leq |J|\leq a_2$.

\end{itemize}
\texttt{Fact 1} implies $\ell \geq n$ for $3\leq n \leq 4$, and since the coloring functions 
\[
(\{v_{1}, e_{23}\},\{v_{3}, e_{12}\},\{v_{2}, e_{13}\}) \mbox{ and } (\{e_{12}, e_{34}\},\{e_{23}, e_{14}\},\{v_{1}, v_{3}\},\{v_{2}, v_{4}\})
\]
are respectively TDCs of $T(C_{3})$ and $T(C_{4})$, we have $\chi^{tt}_d(C_n)=\chi^{t}_d (T(C_n))=n$ for $3\leq n \leq 4$. So we assume $5 \leq n\leq 19$  except $n=15,18$, and continue our proof in the following cases by this assumption that $\mathcal{H}_k$ denotes a graph of order $k$ with positive minimum degree.  

\vspace{0.2cm}

\textbf{Case 1.} $5\leq n \leq 8$. Since for $5 \leq n \leq 8$ the coloring function $g$ with the criterion $g( e_{(i+1)(i+2)})=g(v_{i})=i$ when $1\leq i \leq n$ is a TDC of $T(C_n)$, we have 
\begin{equation}
\label{chi^{t}_d (T(C_n)) leq n, for 5 leq n leq 8}
\chi^{tt}_d (C_n)\leq n ~\mbox{ for }5\leq n \leq 8.
\end{equation}
Since also $\ell\geq n$ for $5 \leq n \leq 8$, by the following reasons, (\ref{chi^{t}_d (T(C_n)) leq n, for 5 leq n leq 8}) implies $\chi^{tt}_d (C_n)=n$. 
\begin{itemize}
\item $n=5$. Let $\ell=4$. Then $(a_1,a_2)=(1,3)$. Because $2a_1+a_2\geq 5$, $a_1+a_2=4$, $a_1=1$ and $\max\{0,5-2a_1,4-a_1\}\leq a_2\leq \min\{ 4, 5-a_1\}$ by \texttt{Facts} $5-8$. But $(a_1,a_2)=(1,3)$ implies $\sum_{i=1}^{4}|V_i|=7\neq 2n$, which contradicts \texttt{Fact 1}. Thus $\ell\geq 5$.

\item $n=6$. Let $\ell=5$. Then $(a_1,a_2)=(1,4)$, $(2,2)$, $(2,3)$. Because $2a_1+a_2\geq 6$, $4\leq a_1+a_2\leq 5$, $1 \leq a_1 \leq 2$ and $\max\{0,6-2a_1,4-a_1\}\leq a_2\leq \min\{ 4, 5-a_1\}$ by \texttt{Facts} $5-8$. Since $(a_1,a_2)=(1,4)$, $(2,3)$ imply $\sum_{i=1}^{5}|V_i|\neq 2n$ and $(a_1,a_2)=(2,2)$ implies $|V_1|>\alpha=4$, which contradict \texttt{Fact 1}, we have $\ell\geq 6$.

\item $n=7$. Let $\ell=6$. Then $(a_1,a_2)=(1,5)$, $(2,3)$, $(2,4)$, $(3,1)$, $(3,2)$, $(3,3)$. Because $2a_1+a_2\geq 7$, $4\leq a_1+a_2\leq 6$, $1 \leq a_1 \leq 3$ and $\max\{0,7-2a_1,4-a_1\}\leq a_2\leq \min\{ 5, 6-a_1\}$ by \texttt{Facts} $5-8$. Since $(a_1,a_2)=(1,5)$, $(2,4)$, $(3,3)$ imply $\sum_{i=1}^{6}|V_i|\neq 2n$ and $(a_1,a_2)=(2,3)$, $(3,1)$, $(3,2)$ imply $|V_1|>\alpha=4$, which contradict \texttt{Fact 1}, we have $\ell\geq 7$.

\item $n=8$. Let $\ell=7$. Then  $(a_1,a_2)=(1,6)$, $(2,4)$, $(2,5)$, $(3,2)$, $(3,3)$, $(3,4)$, $(4,1)$, $(4,2)$, $(4,3)$. Because $2a_1+a_2\geq 8$, $5\leq a_1+a_2\leq 7$, $1 \leq a_1 \leq 4$ and $\max\{0,8-2a_1,5-a_1\}\leq a_2\leq \min\{ 6, 7-a_1\}$ by \texttt{Facts} $5-8$. Since $(a_1,a_2)=(1,6)$, $(2,5)$, $(3,4)$, $(4,3)$ imply $\sum_{i=1}^{7}|V_i|\neq 2n$ and $(a_1,a_2)=(2,4)$, $(3,3)$, $(4,2)$ imply $|V_1|>\alpha=5$, which contradict \texttt{Fact 1}, and $(a_1,a_2)\neq(3,2)$ by \texttt{Fact 10},  we have $(a_1,a_2)=(4,1)$, that is, $(|V_1|,\cdots,|V_{7}|) = (5,5,2,1,1,1,1)$. Since the number of isolate vertices of $T(C_{8})[V_4\cup \cdots \cup V_{7}]$ is at most 1 (by \texttt{Fact 9}) and so $T(C_{8})[V_4\cup \cdots \cup V_{7}]\cong K_1\cup \mathcal{H}_3$, or $\mathcal{H}_4$, we have 
\begin{equation*}
\begin{array}{lll}
|\bigcup_{i=3}^{7} CN(V_i)| & \leq & |CN(V_3)|+|\bigcup_{i=4}^{7} CN(V_i)| \\
 &\leq &2+ \max\{1\times 3 +9,2\times 7\} \\
 &= &|V(T(C_8)|.
\end{array}
\end{equation*}
Since $|\bigcup_{i=3}^{7} CN(V_i)|=|V(T(C_8)|$ if and only if  $T(C_{8})[V_4\cup \cdots \cup V_{7}]\cong 2K_2$, by assumptions $T(C_{8})[V_4\cup V_{5}]\cong K_2$  and $T(C_{8})[V_6\cup V_{7}]\cong K_2$, as you can see in Figure \ref{TDTCC8}, we have
 \begin{itemize}
\item[$\circ$] $|CN(V_3)|=2$ and $|CN(V_4)\cup CN(V_5))|=|(CN(V_6)\cup CN(V_7)|=7$,
\item[$\circ$] $CN(V_3)\cap (\bigcup_{i=4}^{7} CN(V_i))=\emptyset$ and $(CN(V_4)\cup CN(V_5))\cap (CN(V_6)\cup 
CN(V_7))=\emptyset$,
\item[$\circ$] $|V_3 \cap (\bigcup_{i=4}^{5} CN(V_i))|=1$ and $|V_3 \cap (\bigcup_{i=6}^{7} CN(V_i))|=1$.
\end{itemize}
But then the induced subgraph $T(C_{8})[V_1\cup V_{2}]$ with chromatic number 2 contains a clique $K_3$ as a subgraph, which is not possible. Thus $\ell \geq 8$, as desired. 
\end{itemize}
\begin{figure}[ht]\label{TDTCC8}
\centerline{\includegraphics[width=4cm, height=3.8cm]{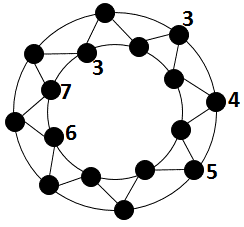}}
\vspace*{0cm}
\caption{The illustration of $T(C_8)$ when $(a_1,a_2)=(4,1)$ and $T(C_{8})[V_4\cup \cdots \cup V_{7}]\cong 2K_2$.} \label{TDTCC8}
\end{figure}
\vspace{0.1cm}

\textbf{Case 2.}  $9 \leq n \leq 19$ except $n=15,18$.

\begin{itemize}

\item $n=9$. Let $\ell=7$. Then $(a_1,a_2)=(2,5)$, $(3,3)$, $(3,4)$, $(4,2)$, $(4,3)$. Because $2a_1+a_2\geq 9$, $6 \leq a_1+a_2\leq 7$ , $2 \leq a_{1} \leq 4$ and $\max\{0,9-2a_1,6-a_1\}\leq a_2\leq \min\{ 5, 7-a_1\}$ by \texttt{Facts} $5-8$. Since $(a_1,a_2)=(2,5)$, $(3,4)$, $(4,3)$ imply $\sum_{i=1}^{7}|V_i|\neq 2n$ and $(a_1,a_2)=(4,2)$, $(3,3)$ imply $|V_1|>\alpha=6$, which contradict \texttt{Fact 1}, we have $\ell\geq 8$. Now since the coloring function, shown in Figure \ref{TDTC7},
\begin{equation*}
(\{v_{1}, v_{6}, v_{8}, e_{23}, e_{45}\},\{v_{7}, v_{9},e_{12}, e_{34}, e_{56}\},\{v_{2}, e_{19}\},\{v_{3}\},\{v_{4}\},\{v_{5},e_{67}\},\{e_{78}\},\{e_{89}\}),
\end{equation*}
is a TDC of $T(C_9)$, we have $\chi^{tt}_d (C_9)= 8$.
\begin{figure}[ht]
\centerline{\includegraphics[width=12cm, height=4.5cm]{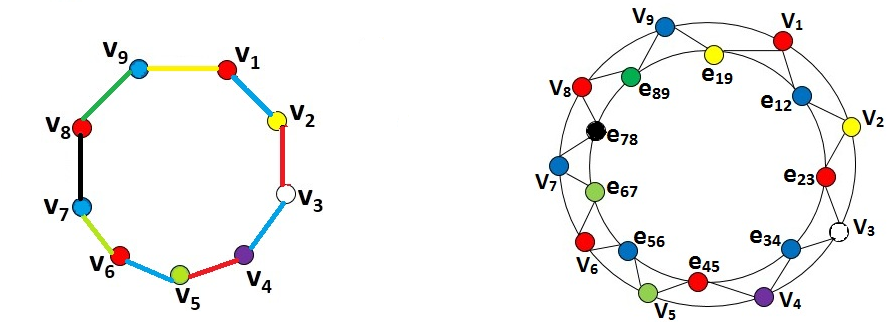}}
\vspace*{-0.25cm}
\caption{ A min-TDTC of $C_{9}$ (left) and the corresponding min-TDC of $T(C_{9})$ (right).}\label{TDTC7}
\end{figure}


\item $n=10$. Let $\ell=8$. Then $(a_1,a_2)=(2,6)$, $(3,4)$, $(3,5)$, $(4,2)$, $(4,3)$, $(4,4)$, $(5,1)$, $(5,2)$, $(5,3)$. Because $2a_1+a_2\geq 10$, $6 \leq a_1+a_2\leq 8$ ,$2 \leq a_{1} \leq 5$ and $\max\{0,10-2a_1,6-a_1\}\leq a_2\leq \min\{ 6, 8-a_1\}$ by \texttt{Facts} $5-8$. Since $(a_1,a_2)=(2,6)$, $(3,5)$, $(4,4)$, $(5,3)$ imply $\sum_{i=1}^{8}|V_i|\neq 2n$ and $(a_1,a_2)=(3,4)$, $(4,3)$, $(5,1)$, $(5,2)$ imply $|V_1|>\alpha=6$, which contradict \texttt{Fact 1}, and $(a_1,a_2)\neq (4,2)$ by \texttt{Fact 10}. Thus $\ell\geq 9$, and in fact $\chi^{tt}_d (C_{10})=9$ by  (\ref{chi^{t}_d (T(C_n))  leq gamma_t(T(C_n))+3}).
\item $n=11$. Let $\ell=9$. Then $(a_1,a_2)=(2,7)$, $(3,5)$, $(3,6)$, $(4,3)$, $(4,4)$, $(4,5)$, $(5,2)$, $(5,3)$, $(5,4)$, $(6,1)$, $(6,2)$, $(6,3)$.  Because $2a_1+a_2\geq 11$, $7 \leq a_1+a_2\leq 9$ ,$2 \leq a_{1} \leq 6$ and $\max\{0,11-2a_1,7-a_1\}\leq a_2\leq \min\{ 7, 9-a_1\}$ by \texttt{Facts} $5-8$.  Since $(a_1,a_2)=(2,7)$, $(3,6)$, $(4,5)$, $(5,4)$, $(6,3)$ imply $\sum_{i=1}^{9}|V_i|\neq 2n$ and $(a_1,a_2)=(3,5)$, $(4,4)$, $(5,3)$, $(6,2)$ imply $|V_1|>\alpha=7$, which contradict \texttt{Fact 1}, and $(a_1,a_2)\neq(4,3)$ by \texttt{Fact 10}, we assume $(a_1,a_2)=(5,2)$, $(6,1)$.
\begin{itemize}

\item[$\circ$] $(a_1,a_2)=(5,2)$. Then, since the number of isolate vertices of $T(C_{11})[V_5\cup \cdots \cup V_{9}]$ is at most 2, (by \texttt{Fact 9}) and so $T(C_{11})[V_5\cup \cdots \cup V_{9}]\cong \overline{K_2}\cup H_{3}$, $K_1\cup H_{4}$ or $H_{5}$, we have
\begin{equation*}
\begin{array}{lllll}
|\bigcup_{i=3}^{9} CN(V_i)| &\leq & |\bigcup_{i=3}^{4} CN(V_i)|+|\bigcup_{i=5}^{9} CN(V_i)|\\
& \leq & 4+\max\{2\times 3+ 9,1\times 3+ 2\times 7,7+9\}\\
& < &|V(T(C_{11}))|,
\end{array}
\end{equation*}
a contradiction with (\ref{CN(V_i)=V}). 

\item[$\circ$] $(a_1,a_2)=(6,1)$. Then the number of isolate vertices of $T(C_{11})[V_4\cup \cdots \cup V_{9}]$ is at most 1 (by \texttt{Fact 9}). If the subgraph has one isolate vertex, then $T(C_{11})[V_4\cup \cdots \cup V_{9}]\cong K_1\cup \mathcal{H}_5$ which implies
\begin{equation*}
\begin{array}{lllll}
|\bigcup_{i=3}^{9} CN(V_i)| &\leq & | CN(V_3)|+|\bigcup_{i=4}^{9} CN(V_i)|\\
& \leq & 2+(3 +7+9)\\
& < &|V(T(C_{11}))|,
\end{array}
\end{equation*}
a contradiction with (\ref{CN(V_i)=V}). So we assume $T(C_{11})[V_4\cup \cdots \cup V_{9}]$ has no isolate vertex. Then $CN(V_i)\cap CN(V_j)\neq \emptyset$ for some $3\leq i < j \leq 9$. Since obviousely $|CN(V_i)\cap CN(V_j)|\geq 2$ implies $|\bigcup_{i=3}^{9} CN(V_i)|< |V(T(C_{11}))|$, we assume $|CN(V_i)\cap CN(V_j)|=1$ for some $3\leq i < j \leq 9$, and so $|\bigcup_{i=3}^{9} CN(V_i)|=|V(T(C_{11}))|$. But then the subgraph $T(C_{11})[V_1\cup V_2]$ with chromatic number 2 contains $K_3$ as a subgraph, which is not possible.  

\end{itemize}
Thus $\ell\geq 10$, and in fact $\chi^{tt}_d (C_{11})=10$ by  (\ref{chi^{t}_d (T(C_n))  leq gamma_t(T(C_n))+3}).


\item $n=12$. Let $\ell=9$. Then $(a_1,a_2)=(3,6)$, $(4,4)$, $(4,5)$, $(5,3)$, $(5,4)$, $(6,2)$, $(6,3)$. Because $2a_1+a_2\geq 12$, $8 \leq a_1+a_2\leq 9$ , $3 \leq a_{1} \leq 6$ and $\max\{0,12-2a_1,8-a_1\}\leq a_2\leq \min\{ 6, 9-a_1\}$ by \texttt{Facts} $5-8$. Since $(a_1,a_2)=(3,6)$, $(4,5)$, $(5,4)$, $(6,3)$ imply $\sum_{i=1}^{9}|V_i|\neq 2n$ and $(a_1,a_2)=(4,4)$, $(5,3)$, $(6,2)$ imply $|V_1|>\alpha=8$, which contradict \texttt{Fact 1}, we have $\ell\geq 10$. Now since $f=(V_{1},V_{2},\cdots,V_{10})$ is a TDC of $T(C_{12})$ where $V_{1}=\{v_{4}, v_{6}, v_{11}, e_{1(12)}, e_{23}, e_{78}, e_{9(10)}\}$, $V_{2}=\{v_{5}, v_{10}, v_{12}, e_{12}, e_{34}, e_{67}, e_{89}\}$, $V_{3}=\{v_{2}\}$, $V_{4}=\{v_{1},v_{3}\}$, $V_{5}=\{e_{45}\}$, $V_{6}=\{e_{56}\}$, $V_{7}=\{v_{7},v_{9}\}$, $V_{8}=\{v_{8}\}$, $V_{9}=\{e_{(10)(11)}\}$, $V_{10}=\{e_{(11)(12)}\}$, we have $\chi^{tt}_d (C_{12})=10$.

\item $n=13$. Let $\ell=10$. Then $(a_1,a_2)=(3,7)$,  $(4,5)$, $(4,6)$, $(5,3)$, $(5,4)$, $(5,5)$, $(6,2)$, $(6,3)$, $(6,4)$, $(7,1)$, $(7,2)$,$(7,3)$. Because $2a_1+a_2\geq 13$, $8 \leq a_1+a_2\leq 10$ , $3 \leq a_{1} \leq 7$ and $\max\{0,13-2a_1,8-a_1\}\leq a_2\leq \min\{ 7, 10-a_1\}$ by \texttt{Facts} $5-8$. Since $(a_1,a_2)=(3,7)$, $(4,6)$, $(5,5)$, $(6,4)$, $(7,3)$ imply $\sum_{i=1}^{10}|V_i|\neq 2n$, and $(a_1,a_2)=(4,5)$, $(5,4)$, $(6,3)$,  $(7,1)$, $(7,2)$ imply $|V_1|>\alpha=8$, which contradict \texttt{Fact 1}, and $(a_1,a_2)\neq (5,3)$ by \texttt{Fact 10}, we assume $(a_1,a_2)=(6,2)$. Then, since the number of isolate vertices of $T(C_{13})[V_5\cup \cdots \cup V_{10}]$ is at most 2 (by \texttt{Fact 9}) and so $T(C_{13})[V_5\cup \cdots \cup V_{10}]\cong \overline{K_2}\cup \mathcal{H}_4$, $\overline{K_1}\cup \mathcal{H}_5$ or $\mathcal{H}_6$,  we have
\begin{equation*}
\begin{array}{lllll}
|\bigcup_{i=3}^{10} CN(V_i)| &\leq & |\bigcup_{i=3}^{4} CN(V_i)|+|\bigcup_{i=5}^{10} CN(V_i)|\\
& \leq & 4+\max\{2\times 3 +2\times 7,3+ 7+9,3 \times 7\}\\
& < &|V(T(C_{13}))|,
\end{array}
\end{equation*}
a contradiction with (\ref{CN(V_i)=V}). Thus $\ell\geq 11$, and in fact $\chi^{tt}_d (C_{13})=11$ by  (\ref{chi^{t}_d (T(C_n))  leq gamma_t(T(C_n))+3}).

\item $n=14$. Let $\ell=10$. Then $(a_1,a_2)=(4,6)$,  $(5,4)$, $(5,5)$, $(6,2)$, $(6,3)$, $(6,4)$, $(7,1)$, $(7,2)$, $(7,3)$. Because $2a_1+a_2\geq 14$, $8 \leq a_1+a_2\leq 10$ , $4 \leq a_{1} \leq 7$ and $\max\{0,14-2a_1,8-a_1\}\leq a_2\leq \min\{ 6, 10-a_1\}$ by  \texttt{Facts} $5-8$. Since $(a_1,a_2)=(4,6)$, $(5,5)$, $(6,4)$, $(7,3)$ imply $\sum_{i=1}^{10}|V_i|\neq 2n$ and $(a_1,a_2)=(5,4)$, $(6,3)$,  $(7,1)$, $(7,2)$ imply $|V_1|>\alpha=9$, which contradict \texttt{Fact 1}, and $(a_1,a_2)\neq (6,2)$ by \texttt{Fact 10}, we have $\ell\geq 11$, and in fact $\chi^{tt}_d (C_{14})=11$ by  (\ref{chi^{t}_d (T(C_n))  leq gamma_t(T(C_n))+3}). 
\item $n=16$. Let $\ell=12$. Then $(a_1,a_2)=(4,8)$,   $(5,6)$, $(5,7)$,  $(6,4)$, $(6,5)$, $(6,6)$, $(7,3)$, $(7,4)$, $(7,5)$, $(8,2)$, $(8,3)$, $(8,4)$, $(9,1)$, $(9,2)$,$(9,3)$. Because $2a_1+a_2\geq 16$, $10 \leq a_1+a_2\leq 12$ , $4 \leq a_{1} \leq 9$ and $\max\{0,16-2a_1,10-a_1\}\leq a_2\leq \min\{ 8, 12-a_1\}$ by \texttt{Facts} $5-8$. Since $(a_1,a_2)=(4,8)$, $(5,7)$, $(6,6)$, $(7,5)$, $(8,4)$, $(9,3)$ imply $\sum_{i=1}^{12}|V_i|\neq 2n$ and  $(a_1,a_2)=(5,6)$,  $(6,5)$,  $(7,4)$,  $(8,3)$, $(9,1)$,$(9,2)$ imply $|V_1|>\alpha=10$, which contradict \texttt{Fact 1}, and $(a_1,a_2)\neq (6,4)$ by \texttt{Fact 10}, we assume $(a_1,a_2)=(7,3)$, $(8,2)$.

\begin{itemize}

\item[$\circ$]  $(a_1,a_2)=(7,3)$. Then, since the number of isolate vertices of $T(C_{16})[V_6\cup \cdots \cup V_{12}]$ is at most 3 (by \texttt{Fact 9}) and so $T(C_{16})[V_6\cup \cdots \cup V_{12}]\cong \overline{K_3} \cup \mathcal{H}_4$, $\overline{K_2}\cup \mathcal{H}_5$, $\overline{K_1}\cup \mathcal{H}_6$ or $\mathcal{H}_7$, we have
\begin{equation*}
\begin{array}{lllll}
|\bigcup_{i=3}^{12} CN(V_i)| &\leq & |\bigcup_{i=3}^{6} CN(V_i)|+|\bigcup_{i=7}^{12} CN(V_i)|\\
& \leq & 6+\max\{3\times 3+ 2\times 7,2\times 3 +7+9,3+ 3\times 7,2 \times 7+9 \}\\
& < &|V(T(C_{16}))|,
\end{array}
\end{equation*}
a contradiction with (\ref{CN(V_i)=V}). 

\item[$\circ$]  $(a_1,a_2)=(8,2)$. Then, since the number of isolate vertices of $T(C_{16})[V_5\cup \cdots \cup V_{12}]$ is at most 2 (by \texttt{Fact 9}) and so $T(C_{16})[V_5\cup \cdots \cup V_{12}]\cong \overline{K_2}\cup \mathcal{H}_6$, $\overline{K_1}\cup \mathcal{H}_7$ or $\mathcal{H}_8$, we have
\begin{equation*}
\begin{array}{lllll}
|\bigcup_{i=3}^{12} CN(V_i)| &\leq & |\bigcup_{i=3}^{4} CN(V_i)|+|\bigcup_{i=5}^{12} CN(V_i)|\\
& \leq & 4+\max\{2\times 3+ 3\times 7,3+2\times 7+9,4\times 7 \}\\
& = &|V(T(C_{16}))|.
\end{array}
\end{equation*}
Since $|\bigcup_{i=3}^{12} CN(V_i)|=|V(T(C_{16}))|$ if and only if $T(C_{16})[V_5\cup \cdots \cup V_{12}] \cong 4K_2$, we may assume $T(C_{16})[V_5\cup \cdots \cup V_{12}][V_{2i-1} \cup V_{2i}]\cong K_2$ for  $3\leq i \leq 6$. Then
 \begin{itemize}
\item[$\circ$] $|CN(V_i)|=2$ for  $3\leq i \leq 4$,
\item[$\circ$] $CN(V_i)\cap (\bigcup_{j=5}^{12} CN(V_j))=\emptyset$ for  $3\leq i \leq 4$,
\item[$\circ$] $|CN(V_{2i-1})\cup CN(V_{2i})|=7$ for $3\leq i \leq 6$,
\item[$\circ$] $(CN(V_{2i-1})\cup CN(V_{2i}))\cap (CN(V_{2j-1})\cup CN(V_{2j}))=\emptyset$ for $3\leq i<j \leq 6$,
\item[$\circ$] $1 \leq |V_3 \cap (\bigcup_{j=5}^{12} CN(V_j))| \leq 2$ and $1 \leq |V_4 \cap (\bigcup_{j=5}^{12} CN(V_j))| \leq 2$. 
\end{itemize}
But this implies that $V_3$, $\cdots$, $V_{12}$ has one of the positions shown in Figure \ref{TDTCC163}. Then, since the induced subgraph $T(C_{16})[V_1\cup V_{2}]$ with chromatic number 2 contains a clique $K_3$ as a subgraph, we reach to contradiction. 

\end{itemize}
\begin{figure}[ht]\label{TDTCC163}
\centerline{\includegraphics[width=13cm, height=3.2cm]{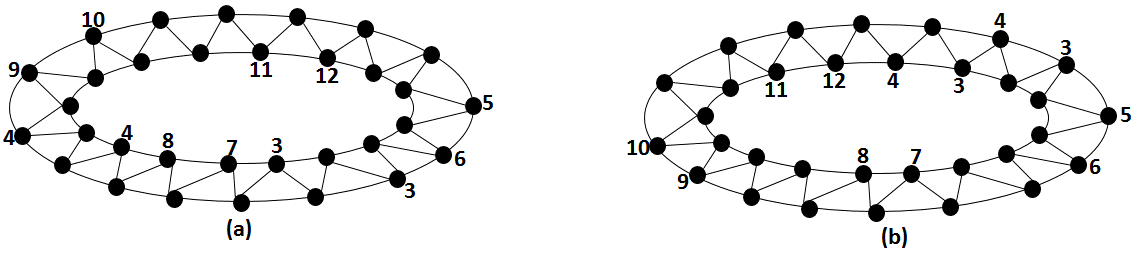}}
\vspace*{0cm}
\caption{The illustration of $T(C_{16})$ when $(a_1,a_2)=(8,2)$ and $T(C_{16})[V_5\cup \cdots \cup V_{12}]\cong 4K_2$.}  \label{TDTCC163}
\end{figure}
\item $n=17$. Let $\ell=12$. Then $(a_1,a_2)=(5,7)$,  $(6,5)$, $(6,6)$, $(7,3)$, $(7,4)$, $(7,5)$, $(8,2)$,$(8,3)$,$(8,4)$,$(9,1)$,$(9,2)$,$(9,3)$. Because $2a_1+a_2\geq 17$, $10 \leq a_1+a_2\leq 12$ and $5 \leq a_{1} \leq 9$ and $\max\{0,17-2a_1,10-a_1\}\leq a_2\leq \min\{ 7, 12-a_1\}$ by \texttt{Facts} $5-8$. Since $(a_1,a_2)=(5,7)$, $(6,6)$, $(7,5)$, $(8,4)$, $(9,3)$ imply $\sum_{i=1}^{12}|V_i|\neq 2n$, and  $(a_1,a_2)=(6,5)$,  $(7,4)$,  $(8,3)$, $(9,1)$,$(9,2)$ imply $|V_1|>\alpha=11$, which contradict \texttt{Fact 1},  and $(a_1,a_2)\neq (7,3)$ by \texttt{Fact 10}, we assume $(a_1,a_2)=(8,2)$.  Then, since the number of isolate vertices of $T(C_{17})[V_5\cup \cdots \cup V_{12}]$ is at most 2 (by \texttt{Fact 9}) and so $T(C_{17})[V_5\cup \cdots \cup V_{12}]\cong \overline{K_2} \cup \mathcal{H}_6$, $\overline{K_1} \cup \mathcal{H}_7$ or $\mathcal{H}_8$, we have
\begin{equation*}
\begin{array}{lllll}
|\bigcup_{i=3}^{12} CN(V_i)| &\leq & |\bigcup_{i=3}^{4} CN(V_i)|+|\bigcup_{i=5}^{12} CN(V_i)|\\
& \leq & 4+\max\{2\times 3+ 3\times 7,3 +2\times 7+9,4 \times 7 \}\\
& < &|V(T(C_{17}))|,
\end{array}
\end{equation*}
a contradiction with (\ref{CN(V_i)=V}). 

Thus $\ell\geq 13$, and in fact $\chi^{tt}_d (C_{17})=13$ by  (\ref{chi^{t}_d (T(C_n))  leq gamma_t(T(C_n))+3}).
\item $n=19$. Let $\ell=14$. Then $(a_1,a_2)=(5,9)$,  $(6,7)$, $(6,8)$, $(7,5)$, $(7,6)$, $(7,7)$,  $(8,4)$, $(8,5)$, $(8,6)$, $(9,3)$, $(9,4)$, $(9,5)$,  $(10,2)$, $(10,3)$, $(10,4)$, $(11,1)$, $(11,2)$, $(11,3)$. Because $2a_1+a_2\geq 19$, $12 \leq a_1+a_2\leq 14$ , $5 \leq a_{1} \leq 11$ and $\max\{0,19-2a_1,12-a_1\}\leq a_2\leq \min\{ 9, 14-a_1\}$ by \texttt{Facts} $5-8$. Since $(a_1,a_2)=(5,9)$, $(6,8)$, $(7,7)$, $(8,6)$, $(9,5)$, $(10,4)$, $(11,3)$ imply $\sum_{i=1}^{14}|V_i|\neq 2n$ and  $(a_1,a_2)=(6,7)$,  $(7,6)$, $(8,5)$, $(9,4)$,  $(10,3)$,  $(11,1)$, $(11,2)$ imply $|V_1|>\alpha=12$, which contradict \texttt{Fact 1}, and $(a_1,a_2)\neq(7,5)$ by \texttt{Fact 10}, we assume $(a_1,a_2)=(8,4)$, $(9,3)$,  $(10,2)$.

\begin{itemize}

\item[$\circ$]   $(a_1,a_2)=(8,4)$. Then, since the number of isolate vertices of $T(C_{19})[V_7\cup \cdots \cup V_{14}]$ is at most 4 (by \texttt{Fact 9}) and so $T(C_{19})[V_7\cup \cdots \cup V_{14}]\cong \overline{K_4}\cup \mathcal{H}_4$, $\overline{K_3}\cup \mathcal{H}_5$, $\overline{K_2}\cup \mathcal{H}_6$, $K_1\cup \mathcal{H}_7$ or $\mathcal{H}_8$, we have
\begin{equation*}
\begin{array}{lllll}
|\bigcup_{i=3}^{14} CN(V_i)| &\leq & |\bigcup_{i=3}^{6} CN(V_i)|+|\bigcup_{i=7}^{14} CN(V_i)|\\
& \leq & 8+\max\{ 4\times 3 +2\times 7,3\times 3 +7+9,2\times 3 + 3\times 7,\\
& &3+ 2\times 7+9,4 \times 7\}\\
& < &|V(T(C_{19}))|,
\end{array}
\end{equation*}
a contradiction with (\ref{CN(V_i)=V}).  

\item[$\circ$]  $(a_1,a_2)=(9,3)$. Then, since the number of isolate vertices of $T(C_{19})[V_6\cup \cdots \cup V_{14}]$ is at most 3 (by \texttt{Fact 9}) and so $T(C_{19})[V_6\cup \cdots \cup V_{14}]\cong \overline{K_3}\cup \mathcal{H}_6$, $\overline{K_2}\cup \mathcal{H}_7$, $K_1\cup \mathcal{H}_8$ or $\mathcal{H}_9$, we have
\begin{equation*}
\begin{array}{lllll}
|\bigcup_{i=3}^{14} CN(V_i)| &\leq & |\bigcup_{i=3}^{5} CN(V_i)|+|\bigcup_{i=6}^{14} CN(V_i)|\\
& \leq & 6+\max\{ 3\times 3 +3\times 7, 2\times 3 + 2\times 7+9, 3+ 4\times 7,\\
&& 3 \times 7+9   \}\\
& < &|V(T(C_{19}))|,
\end{array}
\end{equation*}
a contradiction with (\ref{CN(V_i)=V}). 

\item[$\circ$]  $(a_1,a_2)=(10,2)$. Then, since the number of isolate vertices of $T(C_{19})[V_5\cup \cdots \cup V_{14}]$ is at most 2 (by \texttt{Fact 9}), we have $T(C_{19})[V_5\cup \cdots \cup V_{14}]\cong \overline{K_2}\cup \mathcal{H}_8$, $K_1\cup \mathcal{H}_9$ or $\mathcal{H}_{10}$. Let $T(C_{19})[V_5\cup \cdots \cup V_{14}]\cong \overline{K_2}\cup \mathcal{H}_8$. Then
\begin{equation}
\label{n=19,(10,2)}
\begin{array}{lllll}
|\bigcup_{i=3}^{14} CN(V_i)| &\leq & |\bigcup_{i=3}^{4} CN(V_i)|+|\bigcup_{i=5}^{14} CN(V_i)|\\
& \leq & 4+2\times 3 + 4\times 7\\
& = &|V(T(C_{19}))|. 
\end{array}
\end{equation}

\begin{figure}[ht]\label{TDTCC191}
\centerline{\includegraphics[width=9cm, height=3.8cm]{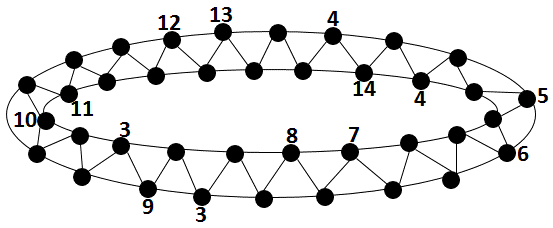}}
\vspace*{0cm}
\caption{The illustration of $T(C_{19})$ when $(a_1,a_2)=(10,2)$, $T(C_{19})[V_5\cup \cdots \cup V_{14}]\cong \overline{K_2}\cup \mathcal{H}_8$ and $|\bigcup_{i=3}^{14} CN(V_i)|=38$.}  
 \label{TDTCC191}
\end{figure}
Since $|\bigcup_{i=3}^{14} CN(V_i)|=38$ if and only if equality holds in (\ref{n=19,(10,2)}), $V_3$, $\cdots$, $V_{14}$ are in the position shown in Figure \ref{TDTCC191}. But then we reach to this contradiction that $T(P_{19})[V_1\cup V_2]$ with chromatic number 2 contains $K_3$ as a subgraph.  Now let $T(C_{19})[V_5\cup \cdots \cup V_{14}]\cong K_1\cup \mathcal{H}_9$. Then
\begin{equation*}
\begin{array}{lllll}
|\bigcup_{i=3}^{14} CN(V_i)| &\leq & |\bigcup_{i=3}^{4} CN(V_i)|+|\bigcup_{i=5}^{14} CN(V_i)|\\
& \leq & 4+1\times 3+ 3\times 7 +9\\
& < &|V(T(C_{19}))|,
\end{array}
\end{equation*}
a contradiction with (\ref{CN(V_i)=V}). Finally let $T(C_{19})[V_5\cup \cdots \cup V_{14}]\cong \mathcal{H}_{10}$, which implies
\begin{equation}
\label{n=19,(10,2)iso=0}
\begin{array}{lllll}
|\bigcup_{i=3}^{14} CN(V_i)| &\leq & |\bigcup_{i=3}^{4} CN(V_i)|+|\bigcup_{i=5}^{14} CN(V_i)|\\
& \leq & 4+5 \times 7\\
& = & 39.
\end{array}
\end{equation}
Since $|\bigcup_{i=3}^{14} CN(V_i)|=38$ if and only if $\mathcal{H}_{10}\cong 5K_2$, by assumptions $\mathcal{H}_{10}[V_{2i-1} \cup V_{2i}]\cong K_2$ for  $3\leq i \leq 7$, we have
\begin{itemize}
\item[$\circ$] $|CN(V_i)|=2$ for  $3\leq i \leq 4$,
\item[$\circ$] $CN(V_i)\cap (\bigcup_{j=5}^{14} CN(V_j))=\emptyset$ for  $3\leq i \leq 4$,
\item[$\circ$] $|CN(V_{2i-1})\cup CN(V_{2i})|=7$ for $3\leq i \leq 7$,
\item[$\circ$] $|(CN(V_{2i-1})\cup CN(V_{2i}))\cap (CN(V_{2j-1})\cup CN(V_{2j}))| \leq 1$ for $3\leq i<j \leq 7$,
 \item[$\circ$] $1 \leq |V_3 \cap (\bigcup_{j=5}^{14} CN(V_j))|\leq 2$ and $1 \leq |V_4 \cap (\bigcup_{j=5}^{14} CN(V_j))| \leq 2$ for some $3\leq j\neq k \leq 6$.
\end{itemize}
That is $V_3$, $\cdots$, $V_{14}$ are in the position shown in Figure \ref{TDTCC194}. But then we reach to this contradiction that $T(P_{19})[V_1\cup V_2]$ with chromatic number 2 contains $K_3$ as a subgraph. 
\end{itemize}
Thus $\ell\geq 15$, and in fact $\chi^{tt}_d (C_{19})=15$ by  (\ref{chi^{t}_d (T(C_n))  leq gamma_t(T(C_n))+3}). 
\end{itemize}
\end{proof}
\begin{figure}[ht]\label{TDTCC194}
\centerline{\includegraphics[width=14cm, height=3cm]{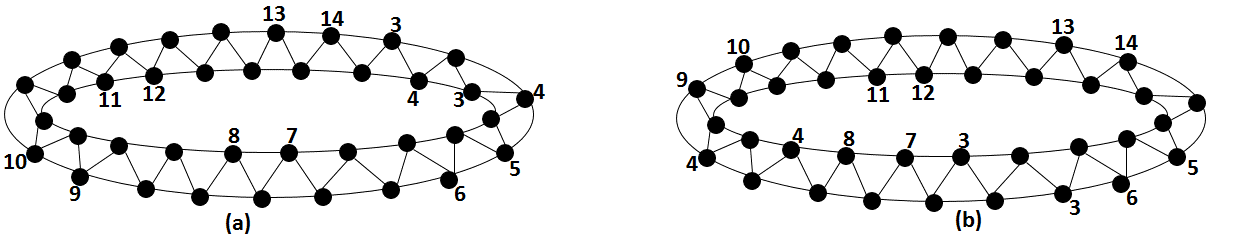}}
\vspace*{0cm}
\caption{The illustration of $T(C_{19})$ when $(a_1,a_2)=(10,2)$, $T(C_{19})[V_5\cup \cdots \cup V_{14}]\cong 5K_2$ and $|\bigcup_{i=3}^{14} CN(V_i)|=38$.}  
 \label{TDTCC194}
\end{figure}

\section{Paths}
Here, we calculate the total dominator total chromatic number of paths. First we recall a proposition and calculate the mixed indepence number of a path. 


\begin{prop} \emph{\cite{KK2017}}
\label{gamma_t T(P_n)} For any path $P_n$ of order $n\geq 2$,
\begin{equation*}
\gamma_{tm}(P_n)=\left\{
\begin{array}{ll}
4\lceil n/7\rceil -3 & \mbox{if }n\equiv 1 \pmod{7},\\
4\lceil n/7\rceil -2 & \mbox{if }n\equiv 2,3,4 \pmod{7},\\
4\lceil n/7\rceil -1 & \mbox{if }n\equiv 5 \pmod{7},\\
4\lceil n/7\rceil  & \mbox{if }n\equiv 0,6 \pmod{7},
\end{array}\right.
\end{equation*}
or equivalently
\begin{equation*}
\gamma_{tm}(P_n)=\left\{
\begin{array}{ll}
\lfloor \frac{4n}{7}\rfloor & \mbox{if }n\equiv 4 \pmod{7},\\
\lceil \frac{4n}{7}\rceil     & \mbox{if }n\not\equiv 4 \pmod{7}.
\end{array}\right.
\end{equation*}
\end{prop}
\begin{lem}\label{alpha (T(P_n))=lfloor 2n-1/3 rfloor}
For any path $P_n$ of order $n\geq 3$, $\alpha_{mix}(P_n)=\lceil \frac{2n-1}{3}\rceil$.
\end{lem}

\begin{proof}
Let  $P_{n}:v_{1}v_{2}\cdots v_{n}$ be a path of order $n\geq 3$. Then $V(T(P_{n}))=V\cup \mathcal{E}$ where $\mathcal{E}=\{e_{i(i+1)}~| ~1\leq i \leq n-1\}$. Since the number of tiangles that a vertex $w$ from an independent set of $T(P_{n})$ belongs to them is 
\begin{equation*}
\begin{array}{ll}
1   & \mbox{ if and only if } w\in \{v_1,v_{n}\} \mbox{ or}  \\
2   & \mbox{ if and only if } w \in \{e_{12},e_{(n-1)n}\} \mbox{ or} \\
3   & \mbox{ if and only if } w \notin \{v_1,e_{12},v_n, e_{(n-1)n} \},
\end{array}
\end{equation*}
we conclude $\{v_{3i+1}~|~ 0\leq i \leq \lceil \frac{n}{3} \rceil -1 \}\cup \{e_{(3i+2)(3i+3)}~|~ 0\leq i \leq \lfloor \frac{n}{3} \rfloor -1  \}$
 is a maximum independent set of cardinality $\lfloor \frac{n}{3}\rfloor +\lceil \frac{n}{3} \rceil=\lceil \frac{2n-1}{3}\rceil$, and so $\alpha_{mix}(P_n)=\alpha (T(P_n))=\lceil \frac{2n-1}{3}\rceil$.
\end{proof}
\begin{prop}
\label{chi^{t}_d T(P_n)}
For any path $P_n$ of order $n\geq 2$,
\begin{equation*}
\chi^{tt}_d (P_n)=\left\{
\begin{array}{ll}
\gamma_{tm}(P_n)+1  & \mbox{if }n=2,3, \\
\gamma_{tm}(P_n)+2  & \mbox{if }n=4,5,6,8,9,10,13,16, \\
\gamma_{tm}(P_n)+3  & \mbox{if }n= 7,\mbox{ } n\neq 13,16\mbox{ or }n \geq 11 
\end{array}
\right.
\end{equation*}
which by considering  Proposition \ref{gamma_t T(P_n)} implies   
\begin{equation*}
\chi^{tt}_d (P_n)=\left\{
\begin{array}{ll}
n+1 & \mbox{if }n=2,\\
n     & \mbox{if }3\leq n \leq 7,\\
n-1  & \mbox{if }8\leq n \leq 9,
\end{array}
\right.
\end{equation*}
and for $n \geq 10$,
\begin{equation*}
\chi^{tt}_d (P_n)=\left\{
\begin{array}{ll}
\lfloor \frac{4n}{7}\rfloor +3 & \mbox{if }n\equiv 4 \pmod{7} \mbox{ or } n=10, 13, 16,\\
\lceil \frac{4n}{7}\rceil   +3   & \mbox{if }n\not\equiv 4 \pmod{7} \mbox{ and } n\neq 10, 13,16.
\end{array}\right.
\end{equation*}
\end{prop}

\begin{proof}
Let $P_{n}:v_{1}v_{2}\cdots v_{n}$ be a path of order $n\geq 2$. Then $V(T(P_{n}))=V\cup \mathcal{E}$ where $\mathcal{E}=\{e_{i(i+1)}~|~ 1\leq i \leq n-1\}$. Since $T(P_2)\cong K_{3}$ and $T(P_3)$ contains $K_3$ as a subgraph and  $(\{v_{1},e_{23}\},\{v_{3},e_{12}\},\{v_{2}\})$ is a TDC of $T(P_3)$, we have $\chi_d^{tt}(P_n)=\chi^{t}_d (T(P_n))=3$ for $n=2,3$. We know from \cite{KK2017} that for $n\geq 4$ the sets
\begin{equation*}
\label{S_r, 0 leq r leq 6}
\begin{array}{ll}
S_0=\{v_{7i+2}, v_{7i+3}, e_{(7i+5)(7i+6)}, e_{(7i+6)(7i+7)}~|~ 0 \leq i \leq \lfloor n/7\rfloor-1\} & \mbox{if }n \equiv 0 \pmod{7}, \\
S_1=S_0 \cup\{e_{(n-1)n}\}& \mbox{if }n \equiv 1 \pmod{7}, \\
S_2=S_3=S_0 \cup\{v_{n-1},v_{n}\} & \mbox{if }n \equiv 2,3 \pmod{7}, \\
S_4=S_0 \cup\{v_{n-2},v_{n-1}\} & \mbox{if }n \equiv 4 \pmod{7}\\
S_5=S_0 \cup\{v_{n-3},v_{n-2},v_{n-1}\} & \mbox{if }n \equiv 5 \pmod{7},\\
S_6=S_0 \cup\{v_{n-4},v_{n-3},e_{(n-2)(n-1)},e_{(n-1)n}\} & \mbox{if }n \equiv 6 \pmod{7}
\end{array}
\end{equation*}
are minimum total dominating sets of $T(P_n)$. Since $\chi (T(P_n)-S_r)=3$ for each $0\leq r\leq 6$, Proposition \ref{chi_d^t =<g_t+min{chi(G-S)}} implies
\begin{equation}
\label{chi^{t}_d (T(P_n))  leq gamma_t(T(P_n))+3}
\chi^{tt}_d (P_n)\leq \gamma_{tm}(P_n)+3.
\end{equation} 
For $n=19$, 22 or $n \geq 24$, since $T(P_n)$ has $2n-3$ triangles with the vertex sets $\{v_i, v_{i+1}, e_{i(i+1)}\}$ when $1\leq i \leq n-1$ or $\{e_{(i-1)i}, v_{i}, e_{i(i+1)}\}$ when $2\leq i \leq n-1$, and since at least two color classes are needed for totally dominating the vertices of every five consecutive triangles with the vertex sets 
\begin{equation*}
\begin{array}{ll}
\{ v_{i}, v_{i+1}, v_{i+2}, v_{i+3}, e_{i(i+1)}, e_{(i+1)(i+2)},e_{(i+2)(i+3)} \} \mbox{ or}  \\
\{v_{i}, v_{i+1},v_{i+2},e_{(i-1)i},e_{i(i+1)},e_{(i+1)(i+2)},e_{(i+2)(i+3)}\},
\end{array}
\end{equation*}
we conclude the number of used color classes in  $T(P_n)$ is at least $2 \lfloor \frac{2n-3}{5} \rfloor \geq \gamma_{t}(T(P_{n}))+3=\gamma_{tm}(P_{n})+3$, and so $\chi_d^{tt}(P_n)=\gamma_{tm}(P_{n})+3 $ by (\ref{chi^{t}_d (T(P_n))  leq gamma_t(T(P_n))+3}). Therefore we continue our proof in the following two remained cases by this assumption that $\mathcal{H}_k$ denotes a graph of order $k$ with positive minimum degree.
\vspace{0.2cm}

\textbf{Case 1.} $n=18,20,21,23$. Let $f=(V_{1},V_{2},\cdots,V_{\ell})$ be a min-TDC of $T(P_{n})$ and let $A=\{v_{1},v_{2},v_{3}\}$, $B=\{v_{n-2},v_{n-1},v_{n}\}$. Then $v_{1} \succ_t V_{k}$ for some $k$ implies $V_{k}=\{w\}$ where $w\in\{v_{2},e_{12}\}$. Let $w \succ_t V_{m}$ for some $m$. Since $V_{m}\subseteq \{v_{1}, v_{3},e_{12},e_{23}\}$ when $w=v_{2}$ and $V_{m}\subseteq \{v_{1}, v_{2}, e_{23}\}$ when $w=e_{12}$, we have $V_m\neq V_k$. Since a similar result holds by considering $v_n$ instead of $v_1$, we conclude $|\{V_{k}~|~ v_{i} \succ_t V_{k} \mbox{ for some } v_{i}\in A\cup B \}|\geq 4$. Let $\{V_{k}~|~ v_i \succ_t V_{k} \mbox{ for some } v_{i}\in A \}=\{V_{1}, V_{2}\}$ and $\{V_{k}~|~ v_i \succ_t V_{k} \mbox{ for some } v_{i}\in B \}=\{V_3, V_4\}$ in which $|V_1|+|V_2|\leq 3$ and $|V_3|+|V_4|\leq 3$. Then $\bigcup_{i=1}^{4}V_i \subseteq A\cup B \cup \{e_{12},e_{23}, e_{(n-2)(n-1)}, e_{(n-1)n}\}$. Now since the subgraph of $T(P_n)-(A\cup B)$ induced by $\{v_4,e_{34}, e_{45}\}$ is a complete graph and  $\{v_4,e_{34},e_{45}\} \cap (\bigcup_{i=1}^{4}V_i) =\emptyset$, we have $\ell\geq 7$. Without loss of generality, we may assume $e_{34} \in V_5$, $e_{45} \in V_6$ and $v_{4} \in V_7$. Let  $C=\{v_i,e_{i(i+1)}~|~6\leq i \leq n-4\}$ and $C_V=\{V_k~|~v \succ_{t} V_k~\mbox{ for some }v \in C\}$. Then for every $v \in C$, $v \nsucc_{t} V_k$  when $1 \leq k \leq 7$. We know the subgraph of $T(P_n)$ induced by $C$ has $2(n-10)$ distinct triangles with the vertex sets $\{v_i, v_{i+1}, e_{i(i+1)}\}$ or $\{e_{(i-1)i}, v_{i}, e_{i(i+1)}\}$ for $6\leq i \leq n-5$. Since at least two color classes are need for totally dominating the vertices of every five consecutive triangles with the vertex sets $\{ v_{i}, v_{i+1}, v_{i+2}, v_{i+3}, e_{i(i+1)}, e_{(i+1)(i+2)},e_{(i+2)(i+3)} \}$ or $\{v_{i}, v_{i+1},v_{i+2},e_{(i-1)i}, e_{i(i+1)},e_{(i+1)(i+2)},e_{(i+2)(i+3)}\}$, we conclude that the number of used color classes in  $T(P_n)[C]$ is at least $2 \lfloor \frac{2(n-10)}{5} \rfloor$. So $\chi_d^t(T(P_n)) \geq 7+2 \lfloor \frac{2(n-10)}{5} \rfloor \geq  \gamma_t(T(P_n))+3$, which implies $\chi_d^{tt}(P_n) =\gamma_{tm}(P_n)+3$ by (\ref{chi^{t}_d (T(P_n))  leq gamma_t(T(P_n))+3}).
\vspace{0.2cm}

\textbf{Case 2.} $4 \leq n \leq 17$.  Let $f=(V_{1},V_{2},\cdots,V_{\ell})$ be a min-TDC of $T(P_{n})$ such that $|V_1|\geq |V_2|\geq \cdots \geq |V_{\ell}|$ and for $1 \leq i \leq \lceil \frac{2n-1}{3}\rceil $, let
\begin{equation*}
A_{i}=\{V_{k}~|~ v \succ_{t} V_{k}\mbox{ and }|V_{k}| = i \mbox{ for some }v \in V\cup \mathcal{E}\}
\end{equation*}
be  a set of cardinality $a_i$ (we recall $\lceil \frac{2n-1}{3}\rceil=\alpha_{mix}(P_n)=\alpha(T(P_n))$ which is denoted simply by $\alpha$). By considering the following facts we countinue our proof in the following subcases.
\begin{itemize}
\item[$\star$] \texttt{Fact 1.} $|V_{k}|\leq \lceil \frac{2n-1}{3}\rceil $ for $1\leq k\leq \ell$ and $\sum_{i=1}^{\ell}|V_i|=2n-1$.
 
\item[$\star$] \texttt{Fact 2.} For any $v \in V\cup \mathcal{E}$, if $v \succ_{t} V_{k}$ for some $1\leq k\leq \ell$, then $|V_{k}|\leq 2$.

\item[$\star$] \texttt{Fact 3.} For any vertex  $v \in M$, if $v \succ_{t} V_{k}$ for some $1\leq k\leq \ell$ and $|V_k|=2$, then $CN(V_k)\cap M=\{v\}$, where $M\in \{\mathcal{E},V\}$, and since $CN(V_k)\cap V\neq \emptyset$ if and only if $CN(V_k)\cap \mathcal{E}\neq \emptyset$, we have $|CN(V_k)|=2$.

\item[$\star$] \texttt{Fact 4.} For any color class $V_k$ of cardinality one, $1\leq |CN(V_k)\cap V|\leq 2$ and $1\leq |CN(V_k)\cap \mathcal{E}|\leq 2$.

\item[$\star$] \texttt{Fact 5.} $a_{1} \geq 2$. Because for $i=1,n$, $v_{i} \succ_{t} V_{k_i}$ implies $|V_{k_i}|=1$ and $V_{k_1}\neq V_{k_n}$.

\item[$\star$] \texttt{Fact 6.} $\gamma_{tm}(P_n)\leq a_1+a_2\leq \ell $. Because the set $S$ is a TDS of $T(P_n)$ where $|S\cap V_i|=1$ for each $V_i\in A_1\cup A_2$ by \texttt{Fact} 2 (for left), and $a_1+\cdots +a_{\alpha}=\ell$ (for right).

\item[$\star$] \texttt{Fact 7.} $2a_1+a_2\geq n$ (by \texttt{Facts} 3, 4).

\item[$\star$] \texttt{Fact 8.} $\max\{n-\ell,2\} \leq a_1 \leq \lfloor \frac{\alpha \ell -2n+1}{\alpha -1} \rfloor$. Because
\begin{equation*}
\begin{array}{llll}
2n-1-a_1             &   =   & |V(T(P_n))|-|\{V_{i} ~|~|V_{i}|=1 \mbox{ for } 1\leq i \leq \ell\}| \\
                          & =     & \Sigma_{|V_i|\geq 2} |V_i|\\
                          & \leq & (\ell-a_1)\alpha
\end{array}
\end{equation*} 
implies the upper bound, and for the lower bound
\begin{equation}
\begin{array}{llll}
2\ell -n             &   \geq    & 2(a_1+a_2)-n & (\mbox{by } \texttt{Fact } 6) \\
                       &  \geq     & a_2                & (\mbox{by }  \texttt{Fact } 7),
     \end{array}
\end{equation} 
implies 
\begin{equation*}
\begin{array}{llll}
a_1             &  \geq   & \frac{n-a_2}{2} & (\mbox{by }  \texttt{Fact } 7) \\
                          & \geq  & n-\ell & \mbox{by } (2.0.2).
     \end{array}
\end{equation*} 

\item[$\star$] \texttt{Fact 9.} $\max\{0,n-2a_1,\gamma_{tm}(P_n)-a_1\}\leq a_2\leq \min\{ 2\ell-n, \ell -a_1\}$ (by \texttt{Facts} $5-8$).

\item[$\star$]  \texttt{Fact 10.} If $J=\{k~|~|V_k|=2 \mbox{ and } |CN(V_k)|\neq 0  \}$, then the number of isolate vertices of $T(P_n)[\bigcup_{|V_k|=1}V_k]$ is at most $|J|$ (because $(|V_k|,|CN(V_k)|)=(2,2)$ implies $T(P_n)[CN(V_k)]\cong K_2$, by \texttt{Fact 3}). 

\item[$\star$]  \texttt{Fact 11.} $V(T(P_n))$ has an unique partition to three maximal independent sets $\mathcal{W}_1$, $\mathcal{W}_2$, $\mathcal{W}_3$ such that 
\begin{equation*}
\begin{array}{ll}
|\mathcal{W}_1|=|\mathcal{W}_2|=|\mathcal{W}_3|=\lceil \frac{2n-1}{3}\rceil   & \mbox{if }n \equiv 2 \pmod{3}, \\
|\mathcal{W}_1|=|\mathcal{W}_2|+1=|\mathcal{W}_3|+1=\lceil \frac{2n-1}{3}\rceil   & \mbox{if }n \equiv 1 \pmod{3}, \\
|\mathcal{W}_1|=|\mathcal{W}_2|=|\mathcal{W}_3|+1=\lceil \frac{2n-1}{3}\rceil   & \mbox{if }n \equiv 0 \pmod{3}.
\end{array}
\end{equation*}
\end{itemize}

\begin{itemize}
\item $n=4,5$. Since $\ell=n-1$ implies $a_{1}=1$, which contradicts \texttt{Fact 5}, we have $\ell \geq n$. Now since $(\{v_2\},\{v_3\},\{e_{12},e_{34}\},\{v_1,e_{23},v_4\})$ is a TDC of $T(P_4)$ and also $(\{v_2\},\{v_3\},\{v_4\},\{v_5,e_{12},e_{34}\},\{v_1,e_{23},e_{45}\})$ is a TDC of $T(P_5)$, we obtain $\chi_d^{tt}(P_n)=n$ when $n=4,5$.  
\item $n=6$. Let $\ell=5$. Then $(a_1,a_2)=(2,2)$, $(2,3)$, $(3,1)$, $(3,2)$. Because \texttt{Facts} $5-9$ imply $2a_1+a_2\geq 6$, $4\leq a_1+a_2\leq 5$, $2 \leq a_1 \leq 3$ and $\max\{0,6-2a_1,4-a_1\} \leq a_2 \leq \min\{ 4, 5-a_1\}$. Since $(a_1,a_2)=(2,3)$, $(3,2)$ imply $\sum_{i=1}^{5}|V_i|\neq 2n-1$ and $(a_1,a_2)=(2,2)$, $(3,1)$ imply $|V_1|>\alpha=4$, which contradict \texttt{Fact 1}, we have $\ell\geq 6$. Now since the coloring function $(\{v_2\},\{v_3\},\{v_4\},\{v_5\},\{e_{12},e_{34},e_{56}\},\{v_1,e_{23},e_{45},v_6\})$ is a TDC of $T(P_6)$, we obtain $\chi_d^{tt}(P_6)=6=n$. 
\item $n=7$. Let $\ell=6$. Then $(a_1,a_2)=(2,3)$, $(2,4)$, $(3,1)$, $(3,2)$, $(3,3)$, $(4,0)$, $(4,1)$, $(4,2)$. Because \texttt{Facts} $5-9$ imply $2a_1+a_2\geq 7$, $4\leq a_1+a_2\leq 6$, $2 \leq a_1 \leq 4$ and $\max\{0,7-2a_1,4-a_1\}\leq a_2\leq \min\{5, 6-a_1\}$.  Since $(a_1,a_2)=(2,4)$, $(3,3)$, $(4,2)$ imply $\sum_{i=1}^{6}|V_i|\neq 2n-1$ and $(a_1,a_2)=(3,2)$, $(4,1)$ imply $|V_1|>\alpha=5$, which contradict \texttt{Fact 1}, we have $(a_1,a_2)=(2,3)$, $(3,1)$, $(4,0)$. 
\begin{itemize}
\item[$\circ$] $(a_1,a_2)=(4,0)$. Then, since the number of isolate vertices of $T(P_7)[V_3\cup \cdots \cup V_6]$ is zero, we have $T(P_7)[V_3\cup \cdots \cup V_6]\cong \mathcal{H}_4$. Since $|\bigcup_{i=3}^{6} CN(V_i)|<|V(T(P_7))|=13$ when $ \mathcal{H}_4\not\cong 2K_2$, we assume $ \mathcal{H}_4\cong 2K_2$. By assumptions $T(P_7)[V_3\cup V_4]\cong T(P_7)[V_5\cup V_6]\cong  K_2$, we have $|(\bigcup_{i=3}^{4} CN(V_i))\cap (\bigcup_{i=5}^{6} CN(V_i))|=1$. This garantees that $(V_3,V_4,V_5,V_6)=(\{v_2\},\{v_3\},\{v_5\},\{v_6\})$ or $(\{v_2\},\{v_3\},\{e_{56}\},\{e_{67}\})$. But then, since $T(P_7)[V_1\cup V_2]$ with chromatic number two contains a complete subgraph $K_3$ with the vetex set $\{v_4,e_{34},e_{45}\}$, we reach to the contradiction.
\item[$\circ$] $(a_1,a_2)=(3,1)$. Then, since the number of isolate vertices of $T(P_7)[V_4\cup \cdots \cup V_6]$ is one and so $T(P_7)[V_3\cup \cdots \cup V_6]\cong K_1\cup K_2$ or $\mathcal{H}_3$, we have 
\begin{equation*}
\begin{array}{lllll}
|\bigcup_{i=3}^{6} CN(V_i)| &\leq & |CN(V_3)|+|\bigcup_{i=4}^{6} CN(V_i)|\\
&\leq &2+\max\{3+7, 9\}\\
&<& |V(T(P_7))|,
\end{array}
\end{equation*}
a contradiction with (\ref{CN(V_i)=V}). 
\item[$\circ$] $(a_1,a_2)=(2,3)$. Then, since the number of isolate vertices of $T(P_7)[V_5\cup V_6]$ is two or zero and so $T(P_7)[V_3\cup \cdots \cup V_6]\cong \overline{ K_2}$ or $K_2$, we have 
\begin{equation*}
\begin{array}{lllll}
|\bigcup_{i=3}^{6} CN(V_i)| &\leq & |\bigcup_{i=3}^{4}CN(V_i)|+|\bigcup_{i=5}^{6} CN(V_i)|\\
&\leq &4+\max\{2 \times 3, 7\}\\
&<& |V(T(P_7))|,
\end{array}
\end{equation*}
a contradiction with (\ref{CN(V_i)=V}). 
\end{itemize}
Thus $\ell\geq 7$, which implies $\chi_d^{tt}(P_7)=7=\gamma_{tm}(P_7)+3$ by (\ref{chi^{t}_d (T(P_n))  leq gamma_t(T(P_n))+3}).
\item $n=8$. Let $\ell=6$. Then $(a_1,a_2)=(2,4)$, $(3,2)$,  $(3,3)$. Because \texttt{Facts} $5-9$ imply $2a_1+a_2\geq 8$, $5\leq a_1+a_2\leq 6$, $3 \leq a_1 \leq 5$ and $\max\{0,8-2a_1,5-a_1\}\leq a_2\leq \min\{4, 6-a_1\}$. Since $(a_1,a_2)=(2,4)$, $(3,3)$ imply $\sum_{i=1}^{6}|V_i|\neq 2n-1$ and $(a_1,a_2)=(3,2)$ imply $|V_1|>\alpha=5$, which contradict \texttt{Fact 1}, we have $\ell\geq 7$. Since now $(\{v_{1}, e_{23}, e_{45}, e_{67}, v_{8}\},\{e_{12},e_{34}, v_{5},e_{78}\},\{v_{4}, e_{56}\},\{v_{3}\},\{v_{2}\},\{v_{6}\},\{v_{7}\})$ is a TDC of $T(P_8)$, we obtain $\chi^{tt}_d (P_{8})=7=n-1$. 
\item $n=9$. Let $\ell=7$. Then $(a_1,a_2)=(2,5)$, $(3,3)$, $(3,4)$, $(4,2)$, $(4,3)$,  $(5,1)$, $(5,2)$. Because \texttt{Facts} $5-9$ imply $2a_1+a_2\geq 9$, $6\leq a_1+a_2\leq 7$, $2 \leq a_1 \leq 5$ and $\max\{0,9-2a_1,6-a_1\}\leq a_2\leq \min\{5, 7-a_1\}$. Since $(a_1,a_2)=(2,5)$, $(3,4)$, $(4,3)$,$(5,2)$ imply $\sum_{i=1}^{7}|V_i|\neq 2n-1$ and $(a_1,a_2)=(3,3)$, $(4,2)$,  $(5,1)$ imply $|V_1|>\alpha=6$, which contradict \texttt{Fact 1}, we have $\ell\geq 8$.  Now since 
\[
(\{v_{2}\},\{v_{3}\},\{e_{45}\},\{e_{56}\},\{v_{7}\},\{v_{8}\},\{v_{1},v_{4},v_{6},v_{9},e_{23},e_{78}\}, \{e_{12},e_{34},e_{67},e_{89},v_{5}\})
\]
 is a TDC of $T(P_9)$, we  have $\chi^{tt}_d (P_9)=8=n-1$.
\item $n=10$. Let $\ell=7$.  Then $(a_1,a_2)=(3,4)$, $(4,2)$, $(4,3)$, $(5,1)$, $(5,2)$. Because \texttt{Facts} $5-9$ imply $2a_1+a_2\geq 10$, $6\leq a_1+a_2\leq 7$, $3\leq a_1\leq 5$ and $\max\{0,10-2a_1,7-a_1\}\leq a_2\leq \min\{4, 7-a_1\}$. Since $(a_1,a_2)=(3,4)$, $(4,3)$, $(5,2)$ imply $\sum_{i=1}^{7}|V_i|\neq 2n-1$ and $(a_1,a_2)=(4,2)$, $(5,1)$ imply $|V_1|>\alpha=7$, which contradict \texttt{Fact 1}, we have $\ell\geq 8$. Now since $(V_1,V_2,\{e_{45}, e_{67}\},\{e_{56}\},\{v_{2}\}, \{v_{3}\}, \{v_{8}\},\{v_{9}\})$ is a TDC of $T(P_{10})$ where $V_1=\{v_{1}, v_{4}, v_{6}, e_{23}, e_{78},e_{9(10)} \}$, $V_2=\{v_{5}, v_{7}, v_{10},e_{12},e_{34}, e_{89} \}$ and  it is shown in Figure \ref{TDTC6}, we have $\chi^{tt}_d (P_{10})=8=\lfloor \frac{4n}{7}\rfloor +3$.
\begin{figure}[ht]
\centerline{\includegraphics[width=7.8cm, height=3.5cm]{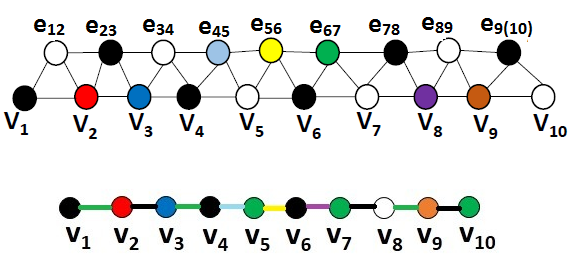}}
\vspace*{-0.25cm}
\caption{A min-TDC of $T(P_{10})$ (up) and the corresponding min-TDTC of $P_{10}$ (down).}\label{TDTC6}
\end{figure}
\item $n=11$. Let $\ell=8$. Then $(a_1,a_2)=(3,5)$, $(4,3)$, $(4,4)$, $(5,1)$, $(5,2)$, $(5,3)$. Because \texttt{Facts} $5-9$ imply $2a_1+a_2\geq 11$, $6\leq a_1+a_2\leq 8$, $3 \leq a_1 \leq 5$ and $\max\{0,11-2a_1,6-a_1\}\leq a_2\leq \min\{5, 8-a_1\}$. Since $(a_1,a_2)=(3,5)$, $(4,4)$, $(5,3)$ imply $\sum_{i=1}^{8}|V_i|\neq 2n-1$ and $(a_1,a_2)=(4,3)$, $(5,2)$ imply $|V_1|>\alpha=7$, which contradict \texttt{Fact 1}, we have $(a_1,a_2)=(5,1)$ that is, $(|V_1|,\cdots,|V_{8}|)= (7,7,2,1,1,1,1,1)$. Then $V_3\cup \cdots \cup V_8$ is a maximal independent set by \texttt{Fact 11}, and all vertices of the set are totally dominated by no color class. Thus $\ell\geq 9$, which implies $\chi_d^{tt}(P_{11})=9=\gamma_{tm}(P_{11})+3$ by (\ref{chi^{t}_d (T(P_n))  leq gamma_t(T(P_n))+3}). 
\item $n=12$. Let $\ell=9$.  Then $(a_1,a_2)=(3,6)$, $(4,4)$, $(4,5)$, $(5,2)$, $(5,3)$, $(5,4)$, $(6,1)$,$(6,2)$, $(6,3)$, $(7,0)$, $(7,1)$, $(7,2)$. Because \texttt{Facts} $5-9$ imply $2a_1+a_2\geq 12$, $7\leq a_1+a_2\leq 9$, $3 \leq a_1 \leq 7$ and $\max\{0,12-2a_1,7-a_1\}\leq a_2\leq \min\{6, 9-a_1\}$. Since $(a_1,a_2)=(3,6)$, $(4,5)$, $(5,4)$, $(6,3)$, $(7,2)$ imply $\sum_{i=1}^{9}|V_i|\neq 2n-1$ and $(a_1,a_2)=(4,4)$, $(5,3)$, $(6,2)$, $(7,1)$ imply $|V_1|>\alpha=8$, which contradict \texttt{Fact 1}, we have $(a_1,a_2)=(5,2)$, $(6,1)$, $(7,0)$. If $(a_1,a_2)=(7,0)$, then $V_3\cup \cdots \cup V_9=\mathcal{W}_3$  which is a maximal independent set 
by \texttt{Fact 11}, and so no vertex of the set is totally dominated by a color class.
\begin{itemize}
\item[$\circ$] $(a_1,a_2)=(5,2)$. 
Then, since  the number of isolate vertices of $T(P_{12})[V_5\cup \cdots \cup V_{9}]$ is at most 2 (by \texttt{Fact 10}) and so $T(P_{12})[V_5\cup \cdots \cup V_{9}]\cong \overline{K_2}\cup H_{3}$, $K_1\cup \mathcal{H}_4$ or $\mathcal{H}_5$, 
\begin{equation*}
\begin{array}{lllll}
|\bigcup_{i=3}^{9} CN(V_i)| &\leq & |\bigcup_{i=3}^{4} CN(V_i)|+|\bigcup_{i=5}^{9} CN(V_i)|\\
& \leq & 4+\max\{ 2\times 3+ 9, 1\times 3 + 2\times 7, 1\times 7+ 1 \times 9\}\\
& < &|V(T(P_{12}))|,
\end{array}
\end{equation*}
a contradiction with (\ref{CN(V_i)=V}). 
\item[$\circ$] $(a_1,a_2)=(6,1)$. 
Then, since  the number of isolate vertices of $T(P_{12})[V_4\cup \cdots \cup V_{9}]$ is at most 1 (by \texttt{Fact 10}) and so $T(P_{12})[V_4\cup \cdots \cup V_{9}]\cong K_1\cup \mathcal{H}_5$ or $\mathcal{H}_6$,
\begin{equation}
\label{P12 (6,1)}
\begin{array}{lllll}
|\bigcup_{i=3}^{9} CN(V_i)| &\leq & |\bigcup_{i=3}^{4} CN(V_i)|+|\bigcup_{i=5}^{9} CN(V_i)|\\
& \leq & 2+\max\{ 1\times 3 + 1\times 7+1\times 9, 3 \times 7\}\\
& = &|V(T(P_{12}))|.
\end{array}
\end{equation}
We see that equality holds in (\ref{P12 (6,1)}) if and only if $T(P_{12})[V_4\cup \cdots \cup V_{9}]\cong 3K_2\cong \bigcup_{i=1}^3 T(P_{12})[V_{2i-1}\cup V_{2i}]$ such that $(CN(V_{2i-1})\cup CN(V_{2i}))\cap (CN(V_{2j-1})\cup CN(V_{2j}))$ for each $1\leq i <j \leq 3$. This implies $V_{4}=\{v_2\}$, $ V_5=\{v_3\}$, $V_{6}=\{e_{56}\}$, $ V_7=\{e_{67}\}$, $V_{8}=\{v_9\}$, $ V_9=\{v_{10}\}$. Since the vertices $v_{12}$ and $e_{(11)(12)}$ did not totally dominated by a color class of cardinality one, we must have $CN(V_3)=\{v_{12},e_{(11)(12)}\}$ which is not possible. 
\end{itemize}
Thus $\ell\geq 10$, which implies $\chi_d^{tt}(P_{12})=10=\gamma_{tm}(P_{12})+3 $ by (\ref{chi^{t}_d (T(P_n))  leq gamma_t(T(P_n))+3}).

\item $n=13$. Let $\ell=9$. Then $(a_1,a_2)=(4,5)$, $(5,3)$, $(5,4)$,  $(6,2)$,$(6,3)$,  $(7,1)$, $(7,2)$. Because \texttt{Facts} $5-9$ imply $2a_1+a_2\geq 13$, $8\leq a_1+a_2\leq 9$, $4 \leq a_1 \leq 7$ and 
$\max\{0,13-2a_1,8-a_1\}\leq a_2\leq \min\{7, 10-a_1\}$. Since $(a_1,a_2)=(4,5)$, $(5,4)$, $(6,3)$, $(7,2)$ imply $\sum_{i=1}^{9}|V_i|\neq 2n-1$, and $(a_1,a_2)=(5,3)$, $(6,2)$, $(7,1)$ imply $|V_1|>\alpha=9$, which contradict \texttt{Fact 1}. 
Thus $\ell\geq 10$. Now since $(V_{1},V_{2},\cdots,V_{10})$ is a TDC of $T(P_{13})$ where
\begin{equation*}
\begin{array}{ll}
V_1=\{v_{1}, v_{6}, v_{8},v_{13},   e_{23}, e_{45},e_{9(10)},e_{(11)(12)} \}, \\
V_2=\{v_{5}, v_{7}, v_{9}, e_{12},e_{34}, e_{(10)(11)},e_{(12)(13)} \}, \\
V_3=\{ v_{4},e_{56}\},~V_4=\{ v_{10},e_{89}\},~V_5=\{ v_2\},~V_6=\{ v_{3}\},\\
V_7=\{ e_{67}\},~V_8=\{ e_{78}\},~V_9=\{v_{11} \},~V_{10}=\{ v_{12}\},
\end{array}
\end{equation*}
we have $\chi^{tt}_d (P_{13})=10$. 
\item $n=14$. Let $\ell=10$. Then $(a_1,a_2)=(4,6)$, $(5,4)$, $(5,5)$, $(6,2)$, $(6,3)$, $(6,4)$, $(7,1)$,$(7,2)$, $(7,3)$. Because \texttt{Facts} $5-9$ imply $2a_1+a_2\geq 14$, $8\leq a_1+a_2\leq 10$, $4 \leq a_1 \leq 7$  and $\max\{0,14-2a_1,8-a_1\}\leq a_2\leq \min\{6, 10-a_1\}$. Since $(a_1,a_2)=(4,6)$, $(5,5)$, $(6,4)$, $(7,3)$ imply $\sum_{i=1}^{10}|V_i|\neq 2n-1$, and $(a_1,a_2)=(5,4)$,  $(6,3)$, $(7,2)$  imply $|V_1|>\alpha=9$, which contradict \texttt{Fact 1}, we have $(a_1,a_2)=(6,2)$,  $(7,1)$, that is, 
\begin{equation*}
\begin{array}{llll}
(|V_1|,\cdots,|V_{10}|)&            =   & (9,8,2,2,1,1,1,1,1,1), \\
                                & \mbox{or} & (9,9,2,1,1,1,1,1,1,1).
\end{array}
\end{equation*}  
If $(|V_1|,\cdots,|V_{10}|) = (9,9,2,1,1,1,1,1,1,1)$, then $V_3\cup \cdots \cup V_{10}$ is a maximal independent set by \texttt{Fact 11}, and so no vertex of the set is totally dominated by a color class by \texttt{Fact 2}.  If $(|V_1|,\cdots,|V_{10}|) = (9,8,2,2,1,1,1,1,1,1)$, then $V_3\cup \cdots \cup V_{10}=\mathcal{W}_j\cup \{w\}$ for some $1\leq j \leq 3$ and some $w\in \bigcup_{j\neq i=1}^3 \mathcal{W}_i$ by \texttt{Fact 11}. Since $deg(w)=4$, there exist at least five vertices in $\mathcal{W}_j$ which are not totally dominated by a color class, a contradiction.

\item $n=15$. Let $\ell=11$. Then $(a_1,a_2)=(4,7)$, $(5,5)$, $(5,6)$, $(6,3)$, $(6,4)$, $(6,5)$, $(7,2)$, $(7,3)$, $(7,4)$, $(8,1)$, $(8,2)$, $(8,3)$, $(9,0)$, $(9,1)$, $(9,2)$. Because \texttt{Facts} $5-9$ imply $2a_1+a_2\geq 15$, $9\leq a_1+a_2\leq 11$, $4 \leq a_1 \leq 9$ and 
$\max\{0,15-2a_1,9-a_1\}\leq a_2\leq \min\{7, 11-a_1\}$. Since $(a_1,a_2)=(4,7)$, $(5,6)$, $(6,5)$, $(7,4)$ , $(8,3)$, $(9,2)$ imply $\sum_{i=1}^{11}|V_i|\neq 2n-1$ and $(a_1,a_2)=(5,5)$, $(6,4)$, $(7,3)$, $(8,2)$, $(9,1)$  imply $|V_1|>\alpha=10$, which contradict \texttt{Fact 1}, we have $(a_1,a_2)=(6,3)$,  $(7,2)$,  $(8,1)$,  $(9,0)$. 
If $(a_1,a_2)=(9,0)$. Then $V_3\cup \cdots \cup V_{11}=\mathcal{W}_3$  is a maximal independent set 
(by \texttt{Fact 11}), and so no vertex of the set is totally dominated by a color class.
\begin{itemize}
\item[$\circ$] $(a_1,a_2)=(8,1)$, that is, $(|V_1|,\cdots,|V_{11}|) =(10,9,2,1,1,1,1,1,1,1,1)$. Then either $V_3\cup \cdots \cup V_{11}=\mathcal{W}_2$  which is a maximal independent set (by \texttt{Fact 11}) and so no vertex of the set is totally dominated by a color class, or $V_3\cup \cdots \cup V_{11}=\mathcal{W}_3\cup \{w\}$ for some vertex $w\in \mathcal{W}_1\cup \mathcal{W}_2$, and since $|N(w)\cap \mathcal{W}_3|\leq 2$, there exist at least seven vertices in $V_3\cup \cdots \cup V_{11}$ which are totally dominated by no color class.

\item[$\circ$] $(a_1,a_2)=(7,2)$. Then, since  the number of isolate vertices of $T(P_{15})[\bigcup_{i=5}^{11}V_i]$ is at most 2 (by \texttt{Fact 10}) and so $T(P_{15})[\bigcup_{i=5}^{11}V_i]\cong \overline{K_2}\cup H_{5}$, $K_1\cup \mathcal{H}_6$ or $\mathcal{H}_7$, we have
\begin{equation*}
\begin{array}{lllll}
|\bigcup_{i=3}^{11} CN(V_i)| &\leq & |\bigcup_{i=3}^{4} CN(V_i)|+|\bigcup_{i=5}^{11} CN(V_i)|\\
& \leq & 4+\max\{ 2\times 3+ 7+9, 1\times 3 + 3\times 7, 2\times 7+ 1 \times 9\}\\
& < &|V(T(P_{15}))|,
\end{array}
\end{equation*}
a contradiction with (\ref{CN(V_i)=V}). 
\item[$\circ$]  $(a_1,a_2)=(6,3)$. Then, since  the number of isolate vertices of $T(P_{15})[\bigcup_{i=6}^{11}V_i]$ is at most 3 (by \texttt{Fact 10}) and so $T(P_{15})[\bigcup_{i=6}^{11}V_i]\cong \overline{K_3}\cup H_{3}$, $\overline{K_2}\cup H_{4}$, $K_1\cup \mathcal{H}_5$ or $\mathcal{H}_6$, we have
\begin{equation*}
\begin{array}{lllll}
|\bigcup_{i=3}^{11} CN(V_i)| &\leq & |\bigcup_{i=3}^{5} CN(V_i)|+|\bigcup_{i=6}^{11} CN(V_i)|\\
& \leq & 6+\max\{ 3\times 3+9, 2\times 3+ 2\times 7, 3 + 7+ 9, 3\times 7\}\\
& < &|V(T(P_{15}))|,
\end{array}
\end{equation*}
a contradiction with (\ref{CN(V_i)=V}). 
\end{itemize}
Thus $\ell\geq 12$, which implies $\chi_d^{tt}(P_{15})=12=\gamma_{tm}(P_{15})+3 $ by (\ref{chi^{t}_d (T(P_n))  leq gamma_t(T(P_n))+3}).
\item $n=16$. Let $\ell=11$.  Then $(a_1,a_2)=(5,6)$,  $(6,4)$, $(6,5)$, $(7,3)$, $(7,4)$, $(8,2)$, $(8,3)$, $(9,1)$, $(9,2)$. Because \texttt{Facts} $5-9$ imply $2a_1+a_2\geq 16$, $10 \leq a_1+a_2\leq 11$, $5 \leq a_1 \leq 9$ and $\max\{0,16-2a_1,10-a_1\}\leq a_2\leq \min\{8, 11-a_1\}$. Since $(a_1,a_2)=(5,6)$, $(6,5)$, $(7,4)$, $(8,3)$ , $(9,2)$ imply $\sum_{i=1}^{11}|V_i|\neq 2n-1$ and $(a_1,a_2)=(6,4)$,  $(7,3)$, $(8,2)$,  $(9,1)$  imply $|V_1|>\alpha=11$, which contradict \texttt{Fact 1}, we have $\ell \geq 12$. Now since $(V_{1},V_{2},\cdots,V_{12})$ is a TDC of $T(P_{16})$ where
\begin{equation*}
\begin{array}{ll}
V_1=\{v_{1}, v_{4}, v_{6},v_{11}, v_{13}, v_{16}, e_{23}, e_{78},e_{9(10)},e_{(14)(15)} \}, \\
V_2=\{v_{5}, v_{7}, v_{10}, v_{12},e_{12},e_{34}, e_{89},e_{(13)(14)},e_{(15)(16)} \}, \\
V_3=\{ v_{2}\},~V_4=\{ v_{3}\},~V_5=\{ e_{45},e_{67}\},~V_6=\{ e_{56}\},~V_7=\{ v_{8}\},\\
V_8=\{ v_{9}\},~V_9=\{e_{(10)(11)},e_{(12)(13)} \},~V_{10}=\{ e_{(11)(12)}\},~V_{11}=\{ v_{14}\},~V_{12}=\{ v_{15}\},
\end{array}
\end{equation*}
we have $\chi^{tt}_d (P_{16})=12$.  
\item $n=17$. Let $\ell=12$. Then $(a_1,a_2)=(5,7)$,  $(6,5)$, $(6,6)$, $(7,3)$, $(7,4)$, $(7,5)$, $(8,2)$,$(8,3)$, $(8,4)$ , $(9,1)$, $(9,2)$, $(9,3)$. Because \texttt{Facts} $5-9$ imply $2a_1+a_2\geq 17$, $10 \leq a_1+a_2\leq 12$, $5 \leq a_1 \leq 9$ and $\max\{0,17-2a_1,10-a_1\}\leq a_2\leq \min\{7, 12-a_1\}$. Since $(a_1,a_2)=(5,7)$, $(6,6)$, $(7,5)$, $(8,4)$, $(9,3)$ imply $\sum_{i=1}^{12}|V_i|\neq 2n-1$ and $(a_1,a_2)=(6,5)$,  $(7,4)$, $(8,3)$,  $(9,2)$  imply $|V_1|>\alpha=11$, which contradict \texttt{Fact 1}, we have  $(a_1,a_2)=(7,3)$,  $(8,2)$,  $(9,1)$. If $(a_1,a_2)=(9,1)$, then $V_3\cup \cdots \cup V_{12}$  is a maximal independent set (by \texttt{Fact 11}) and so no vertex of the set is totally dominated by a color class. Now let $(a_1,a_2)=(8,2)$, that is, $(|V_1|,\cdots,|V_{12}|) =(11,10,2,2,1,1,1,1,1,1,1,1)$. Then $V_3\cup \cdots \cup V_{12}=\mathcal{W}_j\cup \{w\}$ for some $1\leq j \leq 3$ and some  $w\in \bigcup_{j\neq i =1}^3 \mathcal{W}_i$, and since $|N(w)\cap \mathcal{W}_j|\leq 2$, there exist at least nine vertices in $\mathcal{W}_j$ which are totally dominated by no color class. Finally let $(a_1,a_2)=(7,3)$. Then, since  the number of isolate vertices of $T(P_{17})[\bigcup_{i=6}^{12}V_i]$ is at most 3 (by \texttt{Fact 10}) and so $T(P_{17})[\bigcup_{i=6}^{12}V_i]\cong \overline{K_3}\cup H_{4}$, $\overline{K_2}\cup H_{5}$, $K_1\cup \mathcal{H}_6$ or $\mathcal{H}_7$, we have
\begin{equation*}
\begin{array}{lllll}
|\bigcup_{i=3}^{12} CN(V_i)| &\leq & |\bigcup_{i=3}^{5} CN(V_i)|+|\bigcup_{i=6}^{12} CN(V_i)|\\
& \leq & 6+\max\{ 3\times 3+ 2\times 7, 2\times 3+ 7+9, 3 +3\times 7, 2\times 7+9\}\\
& < &|V(T(P_{17}))|,
\end{array}
\end{equation*}
a contradiction with (\ref{CN(V_i)=V}). 
Thus $\ell\geq 13$, which implies $\chi_d^{tt}(P_{17})=13=\gamma_{tm}(P_{17})+3$ by (\ref{chi^{t}_d (T(P_n))  leq gamma_t(T(P_n))+3}).
\end{itemize}
\end{proof}

\section{A Problem}
By comparing the total dominator chromatic numbers of some graphs $G$ such as paths, cycles with their total dominator total chromatic numbers, we see that $ \chi^{tt}_{d}(G)-\chi^{t}_{d}(G) \rightarrow \infty$ when $n\rightarrow \infty$ for them. So, we end our paper with the following important problem. 

\begin{prob}
\label{chi^{t}_{d}(T(G))> chi^{t}_{d}(G)}
Find some real number $r > 1$ such that for any graph $G$, $\chi^{tt}_{d}(G) \geq r \chi^{t}_{d}(G)$.
\end{prob}


\end{document}